\documentclass[11pt]{article}
\usepackage{amsmath,amsfonts,amsthm,amssymb,amscd,colordvi}
\usepackage{authblk}

\usepackage{color}
\usepackage{comment}

\newcommand{\p}{\partial}

\newcommand{\ga}{\gamma}

\newcommand{\si}{\sigma}

\newcommand{\R}{{\mathbb R}}
\newcommand{\F}{{\mathbb F}}

\newcommand{\Z}{{\mathbb Z}}

\newcommand{\N}{{\mathbb N}}

\newcommand{\cC}{{\cal C}}
 
\newcommand{\cF}{{\cal F}}

\newcommand{\cI}{\mathcal I}

\newcommand{\cL}{{\cal L}}

\newcommand{\supp}{\mathop{\rm supp}\nolimits}

\newcommand{\de}{\delta}

\newcommand{\NN}{{\mathbb N \cup \{0\}}}
\newcommand{\II}{\mathcal I}
\newcommand{\lappr}{\lesssim}
\newcommand{\dd}{d_1}

\def\12{\tfrac12}

\def\lan{\langle}
\def\ran{\rangle}

\def\eps{\varepsilon}
\theoremstyle{plain}
\newtheorem{theorem}{Theorem}[section]
\newtheorem{lemma}[theorem]{Lemma}

\newtheorem{proposition}[theorem]{Proposition}
\newtheorem{corollary}[theorem]{Corollary}
\theoremstyle{definition}

\theoremstyle{remark}

\newcommand{\be}{\begin{equation}}
\newcommand{\ee}{\end{equation}}
\newcommand{\lbl}{\label}
\newcommand{\non}{\nonumber}

\newcommand{\qu}{\quad}
\newcommand{\qmb}{\quad\mbox}

\newcommand{\fr}{\frac}

\newcommand{\x}{\mathbf  x}

\newcommand{\z}{\mathbf z}
\newcommand{\vu}{\mathbf u}
\newcommand{\vv}{\mathbf v}

\newcommand{\zz}{\frak z}

\newcommand{\vb}{\mathbf b}
\newcommand{\vc}{\mathbf  c}
\newcommand{\vw}{\mathbf  w}
\newcommand{\va}{\mathbf  a}
\newcommand{\vl}{\mathbf  l}

\newcommand{\ds}{\displaystyle}

\numberwithin{equation}{section}

\title{A refinement of Heath-Brown's theorem on quadratic forms}
\author[1]{Andrey  Dymov}
\author[2]{Sergei Kuksin}
\author[3]{Alberto Maiocchi}
\author[4]{Sergei  Vl\u adu\c t}
\affil[1]{Steklov Mathematical Institute of RAS, Moscow, Russia \\
	\& National Research University Higher School of
	Economics, Russia \\
		\& Skolkovo Institute of Science and Technology, Skolkovo, Russia\thanks{dymov@mi-ras.ru}} 

\affil[2]{Institut de Math\'emathiques de Jussieu--Paris Rive Gauche, CNRS, Universit\'e Paris Diderot, UMR 7586, Sorbonne Paris Cit\'e, F-75013, Paris, France \&  Peoples' Friendship University of Russia (RUDN University),
	6 Miklukho-Maklaya St, Moscow, 117198, Russia}

\affil[3]{Università degli Studi di Milano--Bicocca, Italy \thanks{alberto.maiocchi@unimib.it}}

\affil[4]{Aix Marseille Universit\'e, CNRS, Centrale Marseille, I2M UMR 7373, 13453, Marseille,
	France and IITP RAS, 19 B. Karetnyi, Moscow, Russia \thanks{serge.vladuts@univ-amu.fr}}

\date{}

\begin{document}

\maketitle
\begin{abstract}

In his paper from 1996 on quadratic forms Heath-Brown developed a version of the 
circle method to count points in the intersection of an unbounded quadric with a 
lattice of short period,  if each point is given a weight, and approximated this quantity 
 by the integral of the weight function against  a measure on the quadric. 
 The weight function is assumed to be $C_0^\infty$--smooth and vanish near the singularity of the 
quadric. In our work we allow the weight function to be finitely smooth, not vanish at the singularity  and  have  an explicit decay 
at infinity.

The paper uses only elementary  number theory and is available to readers without a number-theoretical
background.

\end{abstract}
\section{Introduction}

\subsection{Setting and result} \label{s_1.1}

Let us consider a  non-degenerate quadratic form with integer coefficients on $\R^d$,  $d\geq 4$,
\be\label{FA}
F(\z) = \tfrac12 A \z\cdot \z\ , 
\ee
  which implies that $A$ can be chosen as a non--degenerate 
symmetric matrix with integer elements whose  diagonal elements are even.  If $F$ is sign--definite, then 
for $t\in \R$  the quadric 
\be\label{st}
\Sigma_t =\{ \z: F^t(\z)=0\}, \quad  F^t := F-t, 
\ee
is either an ellipsoid or an empty set, while in the non sign--definite
case $\Sigma_t$ is an unbounded hyper-surface in $\R^d$, which is
smooth if $t\ne0$, while $\Sigma_0$ is a cone and  has a 
  singular point at
zero.  

Let $\Z^d_L$ be the lattice of a small period $L^{-1}$,
$$
\Z^d_L =L^{-1} \Z^d, \qquad    L\ge 1,
$$
and let $w$ be a {\it  regular } real function on $\R^d$ which means that $w$ and its Fourier transform $\hat w(\xi)$ are continuous functions which 
 decay at infinity {  sufficiently fast:
\be\label{regular}
|w(\z)| \le C  |\z|^{-d-\gamma}\ ,\quad\forall \z\in
  \R^d\ , \qquad |\hat w( \mathbf \xi)| \le C
|\xi|^{-d-\gamma}\ , \quad\forall \xi\in \R^d\ , 
\ee
for some $\gamma$, $C>0$, where $|\cdot|$ denotes the Euclidean
  norm. }Our goal is to study the behaviour of  series 
$$
N_L(w; A, m):=  \sum_{\z \in \Sigma_m \cap \Z^d_L} 
      w(\z)\,, 
$$
where $m\in \R$ is such that   ${L^2m}$ is an integer.\footnote{E.g., $m=0$ -- this case is the most important for us.} Let
$$
w_L(\z) := w(\z/L).
$$
Then, obviously,  
\be\label{obvious}
N_L(w{ ; A, m}) = N_1(w_L{ ;A, L^2m}) =: N(w_L{ ;A, L^2m}) \,.  
\ee
We will also write $N_L(w;A):=N_L(w;A,0)$ and $N(w_L;A):= N(w_L;A,0)$.
To study $N_L(w{ ;A, {m}})$ we closely follow the circle method in the form, given to it by  Heath-Brown in \cite{HB}. 
Our notation differs a bit from that in  \cite{HB}. Namely, under the scaling 
$z=z'/L$, $z'\in \Z^d$, we count (with weights) solutions of equation
$F(z') = mL^2$, 
$z'\in \Z^d,$
while Heath-Brown writes the equation as
$
F(z') = m$, $z'\in \Z^d,
$
so that his $m$ corresponds to our $L^2m$.

 We  start with a key theorem which expresses  the analogue of Dirac's
delta function on integers, i.e.  the function 
 $\delta:\Z\to\R$ such that
\begin{equation*}
  \delta(n):= \left\{ \begin{array}{cc}
    1 & \mbox{for } n=0\\
    0 & \mbox{for } n\neq 0 \end{array} \right. \,,
\end{equation*}
through a sort of Fourier representation. This result goes back at least  to Duke, Friedlander and Iwaniec \cite{DFI} (cf. also \cite{I}) , and we  state it in the form, given in
 \cite[Theorem~1]{HB}; basically, it replaces (a major arc decomposition of) the trivial identity
$\delta(n)= \int_0^1 e^{2\pi i\alpha n}d\alpha $
employed in the usual circle method. In
the theorem below for  $q\in { \N}$ we denote by  $e_q$ the exponential function  $e_q(x):= e^{\tfrac{2\pi i x}{q}}$, and 
denote by ${\sum\limits_{a(\operatorname{mod}
	q)}}^* $ the summation over residues $a$ with
$(a,q)=1$, i.e., over all integers  $a\in [1, q-1]$, relatively prime with $q$. 
  
\begin{theorem}

\label{th:1}
  For any $Q\geq1$, there exists $c_Q>0$ and a smooth function 
  $ h(x,y): \R_{>0}\times
  \R\to \R$,  such that
  \begin{equation}\label{gl}
    \delta(n) = c_QQ^{-2}\sum_{q=1}^{\infty}{\sum_{a(\!\!\!\!\!\!\mod
        q)}}^* e_q(an) h\left(\frac qQ,\frac n{Q^2}\right)\,.
  \end{equation}
  The constant $c_Q$ satisfies
  $\,
c_Q=1+O_N(Q^{-N})\,
$
for any $N>0$, while $h$ is such that 
$h(x,y)\le c/x$ and  $h(x,y)=0$ for $x>\max(1,2|y|)$ {  (so for each $n$ the sum in \eqref{gl} contains finitely many non zero terms). }
\end{theorem}

Since for any function $\tilde w$ on $\R^d$ the quantity
$N(\tilde w;A, t)$ may be written as
$
\sum_{\z\in\Z^d} \tilde w(\z) \delta (F^t(\z)),
$
then Theorem~\ref{th:1} allows to represent the series $N(\tilde w;A, t)$ as an iterated sum. Transforming that sum  further using the Poisson summation 
formula as in \cite[Theorem~2]{HB} we arrive at the following result:\,\footnote{In \cite{HB} the result below is stated for $\tilde w\in C_0^\infty$. But the argument there,
based on the Poisson summation,  applies as well to regular functions $\tilde w$.}

\begin{theorem}[Theorem 2 of \cite{HB}]\label{th:2}
  For any regular function $\tilde w$, any $t$ and any 
  $Q\geq 1$ 
  we have the expression
  \begin{equation}\label{eq:sum}
    N({\tilde w}{ ; A, t})=c_QQ^{-2}\sum_{\vc\in
      \Z^{d}}\sum_{q=1}^\infty q^{- d} 
    S_q(\vc) I^0_q(\vc)\,,
  \end{equation}
  with
  \begin{equation}\label{eq:S}  
 S_q(\vc)=S_q(\vc;A,t) :={\sum_{a(\!\!\!\!\!\!\mod
        q)}}^* \sum_{\vb(\!\!\!\!\!\!\mod q)} e_q(a{ F^t}(\vb) + \vc\cdot \vb)    
  \end{equation}
and
\begin{equation}\lbl{I^0_q}  
I^0_q(\vc)=I^0_q(\vc; A,t,Q) := \int_{\R^{d}} {\tilde w}(\z)h\left(\frac qQ,
\frac{{ F^t}(\z)}{Q^2}\right) e_q(-\z\cdot \vc)
\,d\z\,.
\end{equation}
\end{theorem}

We will apply  Theorem~\ref{th:2} to examine  for large $L$ the sum
$N(w_L{ ;A, L^2 m})$ $ = N_L(w{ ;A, m})$,  choosing 
  $\tilde w=w_L$, $t=L^2m$ and $Q=L\geq1$ and estimating  
explicitly the leading terms in $L$ of $S_q(\vc)$ and $I^0_q(\vc)$ as well as the remainders. The answer will be given in terms of  the 
 integral
\be\label{lead_term}
\sigma_{\infty}(w)=
\sigma_{\infty}(w; A, { t}) = \int_{ \Sigma_t}
w(\z)\,\mu^{ \Sigma_t}(d\z)\ 
\ee
(which is  singular if $t=0$). Here 
$
\mu^{\Sigma_t}(d\z)=|\nabla F(\z) |^{-1}dz|_{\Sigma_t}= |A\z|^{-1}dz|_{\Sigma_t},
$
with
$dz|_{\Sigma_t}$ 
representing the volume element over $\Sigma_t$, induced from the standard euclidean structure on $\R^{d}$,
and $A$ the symmetric matrix in \eqref{FA}. 
For regular functions $w$ this integral converges (see Section~\ref{sec:7}).

To write down the asymptotic for $N_L(w{ ;A, m})$ we will need the following quantities, where $p$ ranges over all primes and $c\in\Z^d$:
\begin{equation}\label{eq:sigma_p}
\sigma_p^\vc=\sigma_p^\vc(A,{L^2m}):=\sum_{l=0}^\infty p^{-dl}S_{p^l}(\vc; A,{L^2m}),
\qquad \sigma_p:=\sigma_p^{\mathbf 0},
\end{equation}
where $S_1\equiv 1$,
$$
\sigma^*_\vc(A):=\prod_p (1-p^{-1})\sigma_p^\vc(A,0), \qquad \sigma^*(A):=\sigma^*_{\mathbf 0}(A)=\prod_p(1-p^{-1}) \sigma_p(A,0),
$$ 
and 
\be\label{77}
\sigma(A,{L^2m}) =\prod_p\sigma_p^{\mathbf 0}(A,{L^2m})= \prod_p\sigma_p(A,{L^2m}).
\ee
The products in the formulas above are taken over all primes.
 In the asymptotics, where these quantities are used, they are bounded
 uniformly in $L$ (see Theorems~\ref{th:3} and \ref{th:3-4}, as well
 as Proposition~\ref{p15}).

Everywhere  below for  a function $f\in C^k(\R^d)$ we denote 
$$
\| f\|_{n_1, n_2}  = \sup_{\z\in \R^d} \max_{\left|\alpha\right|_1\le n_1} |\p^\alpha
f(\z)| \lan \z\ran^{n_2}\,,
$$
where   $n_1 \in \NN$, $n_1\le k,$ and  $ n_2\in  \R$.   Here
\begin{equation*}
  \langle \x\rangle := \max\{1,|\x|\} \quad \text{for}\;\: \x\in \R^l, \qu l\in \N,
\end{equation*}
and  $|\alpha|_1 \equiv \sum\alpha_j $ for any integer
vector 
$
\alpha \in (\N\cup\{0\})^d.
$
By  $\cC^{n_1,n_2}(\R^d)$ we denote  a linear space of
$C^{n_1}$-smooth functions $f:\R^d\to \R$, satisfying
$\|f\|_{n_1,n_2}<\infty$. 

{ 
 Note that if $w\in \cC^{d+1,d+1}( \R^d)$ then the function $w$ is regular, so Theorem~\ref{th:2} applies. Indeed, the first relation in \eqref{regular} is obvious. To
prove the second note that for any integer vector 
$
\alpha \in (\N\cup\{0\})^d,
$
$
\xi^\alpha \hat w(\xi) = \big( \frac{i}{2\pi}\big)^{\left|\alpha\right|_1} \widehat{
\p_\x^\alpha w}(\xi).
$
But if 
$
\left|\alpha\right|_1  \le d+1,
$
then $|\p_\x^\alpha w| \le C \lan \x\ran^{-d-1}$, so $\p_\x^\alpha w$ is an $L_1$-function. Thus its Fourier transform 
 $ \widehat{\p_\x^\alpha w}$ is a bounded continuous function for each 
$\left|\alpha\right|_1 \le d+1$ and the second relation in \eqref{regular} also holds. }

Now we formulate our main results. 
First we treat the case $d\geq 5$.

\begin{theorem}\label{th:3}
 Assume that  $d\ge5$. Then for    any $0<\eps\le1$
    there exist positive constants $K_1(d,\eps)$, $K_2(d,\eps)$ and
    $K_3(d,\eps)$, with $K_2(d,\eps)\le K_3(d,\eps)$,
  such that  if  $w \in \cC^{K_1,K_2}(\R^{d})\cap \cC^{0,K_3}(\R^{d})$ and a real number $m$ satisfies ${L^2m} \in \Z$,  then 
  \begin{equation}\label{m_result}
  \big|   N_L(w{ ; A, m})-
  \sigma_\infty(w)\sigma(A,    L^2m) L^{d-2} \big|  \le C 
    L^{d/2+\eps}\left(\|w\|_{K_1,K_2}+\|w\|_{0,K_3}\right),
  \end{equation}
  where the constant $C$ depends on $d, \eps$, $m$ and  $A$. 
 The constant $\sigma(A,L^2m)$ is bounded uniformly in $L$ and $m$.
  In particular if  $\eps=1/2$, then one can take  $K_1= 2d(d^2+d-1)$, $K_2=4(d+1)^2+3d+1$ and $K_3=K_1+3d+4$.
\end{theorem}

Next we study the case $d=4$, restricting ourselves for the situation when $m=0$.
\begin{theorem}\label{th:3-4} Assume that  $d=4$ and $m=0$.  Then for    any $0<\eps < 1/5$
    there exist positive constants $K_1(\eps)$ and  $K_2(\eps) $, 
  such that for  $w \in \cC^{K_1,K_2}(\R^d)$
  \begin{equation}\label{m_result-4}
    \begin{split}
      \bigl|   N_L(w ;A,0 )-
\eta(0)  \sigma_\infty(w)\sigma^*(A) L^{d-2}\log L &-
\sigma_1(w;A,L)L^{d-2} \bigr| \\
&\le  C_0 
    L^{d-2-\eps}\|w\|_{K_1,K_2},
    \end{split}
  \end{equation}
  where the constant $C_0$ depends on $ \eps$ and $A$. 
  The constant $\eta(0)$ is 1 if the determinant $\det  A$ is a square of
  an integer and is 0 otherwise.
   The $L$-independent constant $\sigma^*(A)$ is finite while the constant 
   $\sigma_1$ satisfies
   $$
   |\sigma_1(w;A,L)| \leq C_0 \|w\|_{K_1,K_2}
   $$
   uniformly in $L$.
 	   In the case of a square determinant $\det A$, when $\eta(0)=1$, it is given by  \eqref{def:sigma_1}. 
 	In the case of a non--square determinant $\det A$, when $\eta(0)=0$ and the term
  $\sigma_1(w;A,L)L^{d-2}$ gives the asymptotic of the sum $N_L$, the constant $\sigma_1(w;A,L)$ does not depend on $L$ and
    has  the form  
  \be\lbl{non-sq}
  \sigma_1(w;A)=
  \sigma_\infty(w)L(1,\chi)\prod_p(1-\chi(p)p^{-1})\sigma_p(A,0)  \,,
  \ee
  where  $\chi$ is the Jacobi symbol $(\tfrac{\det(A)}{*})$
  and $L(1,\chi)$ is the Dirichlet $L$--function.\footnote{Concerning the classical notion of the Jacobi symbol and the Dirichlet $L$-function we refer a reader without 
 number-theoretical background e.g. to \cite{Se} and  \cite{Ka}.}
\end{theorem}

If $\eta(0) \sigma^*(A) =0$,  then the asymptotic \eqref{m_result-4} degenerates.
 Similar \eqref{m_result} also 
degenerates to an upper bound on $N_L$, 
 unless we know that $\sigma(A,L^2m)$ admits a suitable positive lower
 bound, for all $L$. Luckily enough,
the required  lower bounds often exist, see below Proposition~\ref{p15}.

Theorems \ref{th:3} and \ref{th:3-4} refine Theorems~5, 6 and 7 from  \cite{HB}    in three respects: 
firstly, now the weight function $w$ has  finite  smoothness and sufficiently fast 
 decays at infinity, while in \cite{HB} $w\in C_0^\infty$. Secondly, we specify how the remainder depends on $w$. 
 Thirdly and the most importantly, we remove the restriction that the support of $w$ does not contain the origin, imposed in \cite{HB} 
 in a number of crucial statements.
These improvements are crucial for us since in our work  \cite{DKMV},
dedicated to the problem of wave turbulence, the two theorems above  are used in the
situation when $w(0) \ne0$  and the support of $w$ is not compact.  A similar  specification of the Heath-Brown method was obtained in 
 \cite[Section~5]{BGHS18} to study an averaging  problem, related to the questions,    considered in \cite{DKMV}.
 Apart from wave turbulence and averaging, the replacement of sums over integer points of a quadric by integrals, 
 with careful estimating the remainders, is needed in Kolmogorov
 Arnold Moser theory  for partial differential equations, e.g. see (C.2) in \cite{EGK}. 
 The publications 
 \cite{EGK, BGHS18, DKMV} are recent. We are certain that these days, when people, working in PDEs and dynamical systems, treat complicated 
 non-linear phenomena with resonances more and more often, there will be increasing demand for the instrumental 
 asymptotics \eqref{m_result}, \eqref{m_result-4} and their variations. Our paper uses only basic results 
 from the number theory and is well available to readers from Analysis. 
 \smallskip

 We note  that the  papers \cite{G18} and \cite{T19} treat the sums $N_L(w; A, m)$  for even and odd dimensions
  $d$ correspondingly, without the restriction that   $w(0)\ne 0$, 
   in a  more general context than  our Theorems~\ref{th:3} and \ref{th:3-4}. 
However, due to this generality the corresponding constants in the asymptotical in $L$ 
 formulas  in  \cite{G18}  and \cite{T19}  are very implicit (e.g., the question whether
they vanish or not  is highly non-trivial).   The connection of the constants 
with singular integrals like \eqref{lead_term} and the dependence of the remainders in the asymptotics on the weight function $w$, 
 crucial for application in analysis,   is not clear.  Another feature of \cite{G18, T19}
 is the use of rather advanced adelic technique, which makes it difficult for readers without  serious 
 number-theoretical background to use the result and the method of the work.

\smallskip

\noindent{\it Remarks.}
1) Theorem~\ref{th:3} is a refinement of Theorem~5 of \cite{HB},
while Theorem~\ref{th:3-4} refines Theorems 6 and 7 of \cite{HB}.
In \cite{HB} also is available some asymptotic in $L$  information about
behaviour of the sums $ N_L(w{ ; A, m})$ when  $d=4$, $m\neq 0$ and
$d=3$, $m=0$.  Since our proof of Theorems~\ref{th:3} and \ref{th:3-4} is based 
on ideas from \cite{HB},  strengthened by Theorem~\ref{c_integrals}, which is valid for $d\ge3$, then most likely 
our approach allows to generalise the above-mentioned results of   \cite{HB}  for $d=3,4$ 
to the case when $w\in  \cC^{K_1,K_2}(\R^d)$ 
with suitable $K_1, K_2$.

\noindent 2) In our work  the dependence of constants in estimates on $m$ is uniform
  on  compact intervals,  while the dependence on the operator $A$ is only via   the norms of $A$ and $A^{-1}$. 
  
  \noindent 3)
  The values of  constants $K_j(d, \eps)$ in \eqref{m_result}, given in Theorem~\ref{th:3}, 
   are far from optimal since   it was not our goal to optimise them. 
    \smallskip
  
    \noindent   4) As the theorems' proof are based on the representation \eqref{eq:sum}, then the function $w$ should be regular
    (see \eqref{regular}). But this holds true if 
    $w\in\cC^{d+1,d+1}$
    and so is valid  if the constants $K_1, K_2 $ are  sufficiently big. E.g. if $K_1, K_2 $
    are  as big as in the last line of the assertion of Theorem~\ref{th:3}.
      
  \medskip
  
 {\it \noindent Brief discussion of the proofs.} We present  in full only a proof of Theorem~\ref{th:3}, which  resembles  
that of \cite[Theorem~5]{HB} with an additional control of how the constants depend
on $w$.   The significant difference from the argument of Heath-Brown  comes in Sections~3 and 4, 
 where we do not assume that the function $w$ vanishes near the
origin, the last assumption being crucial in the analysis of integrals
in Sections~6 and 7 of \cite{HB}. To cope with this difficulty,  which becomes apparent e.g.
in Proposition~\ref{l:I_q(0)=}
	below,  we have to  examine the smoothness at zero of  function 
  \be\label{hrum}
  t \mapsto \sigma_\infty (w;A,  t)
  \ee
  and its decay at  infinity. The corresponding analysis is performed in Section~\ref{sec:7}. There, 
  using the techniques, developed in \cite{DK1} to study  
  integrals \eqref{lead_term}, we prove that  function \eqref{hrum} is $({ \lceil d/2\rceil}-2)$-smooth, 
  but in general for even $d$  
  its  derivative of  order  $(d/2-1)$ may have a logarithmic 
  singularity at zero. There we also  estimate  the rate of  decay of   function \eqref{hrum} at infinity. 

The proof of Theorem~\ref{th:3-4} resembles  that of
Theorems~6 and 7 of \cite{HB} with a new addition given by
Proposition~\ref{l:I_q(0)=}, based on the result of Section~\ref{sec:7}.  
We thus limit ourselves to a sketch of the theorem's  demonstration, given in  Section~\ref{sec:2-4}
in parallel to that of Theorem~\ref{th:3}, and point out the main
differences between the two proofs.   Establishing 
Theorem~\ref{th:3-4} we use certain results from \cite{HB} (namely, Lemmas~30 and 31) without  proof.
  \medskip
  
 \noindent
{\it  Lower bounds for  constant  from the asymptotics.} 
 Let us now discuss lower bounds for the constants $\sigma(A,    L^2m)$ and $\sigma^*(A)$ 
 from Theorems~\ref{th:3} and  \ref{th:3-4}.
  \begin{proposition}\label{p15}
  		$(i)$  If $d\ge 5$ then there exist
      positive constants $c(A)<C(A)$ such that $0<c(A)\le \sigma(A,    L^2m)\le C(A)<\infty $ for any non-degenerate matrix $A$, uniformly in $L$ and $m$.
  	
  	$(ii)$  If $d=4$ and $m=0$ we have $\sigma^*(A)>0 $ for any non-degenerate matrix $A$ such that the corresponding equation $2F({\bf z})=A {\bf z}\cdot {\bf z}=0$   has non-trivial solutions in every p-adic field (in particular this holds if the equation has a non-trivial solution in ${\Z^4}$). 
  \end{proposition}
  
  See Theorems 4, 6 and 7 of \cite{HB}.  
  We do not prove this result, but  just note that its  demonstration uses a refinement of the calculation in the 
  second part of the proof of Lemma~\ref{l:31HB}. Namely, while the lemma 
  gives an upper bound for the desired quantity, a more thorough analysis  permits also to establish the claimed lower bounds. 

   In Appendix B we give essentially a
   complete calculation,  proving Proposition \ref{p15} in the case of the simplest quadratic form $F=\Sigma_{i=1}^{d/2} x_iy_i$, $ d=2s\ge 4$,     and $m=0$. A proof of the proposition for  any $A$
    may follow the same lines,   replacing  explicit formulas by some general results (e.g.  Hensel's Lemma).
  \medskip

   \noindent 
 {\it Non-homogeneous quadratic polynomials. }Now
 consider a non-homogeneous quad\-ratic  polynomial $\cF$ with the second order part, equal to $F$ in \eqref{FA}:
 $$
 \cF(\z) = \tfrac12 A\z\cdot \z +\z_* \cdot\z +\tau, \qquad \z_*\in \R^d,\; \tau \in\R,
 $$
 and the corresponding  set $\Sigma^\cF =\{\z: \cF(\z) =0\}$, 
 $
 N_L(w; \cF) = \sum_{\z \in \Sigma^\cF\cap \Z^d_L} w(\z). 
 $
 Denote 
 $$
 \zz =A^{-1}\z_*, \quad \z' =\z+\zz, \quad m = \tfrac12 \zz\cdot A\zz- \tau,
 $$
 and assume that $\zz\in \Z^d_L$ \footnote{This holds e.g. if  $\det A=\pm1$ and $\z_*\in \Z^d_L$.}  and $ L^2 \tau \in\Z$.
 Then ${L^2m}\in\Z$, $\z'\in\Z^d_L$ if and only if $\z \in\Z^d_L$, and $\cF(\z) = F(\z') -m$. So setting 
 $
 w^\zz(\z') = w(\z' -\zz)
 $
 we have 
 $
 N_L(w; \cF) = N_L(w^\zz; A,m).
 $
 Since
 \[
 \begin{split}
 \sigma_\infty(w^\zz; A,m)=  \int_{\Sigma_m} w^\zz(\z') \frac{d\z'\!\mid_{\Sigma_m}}{ |\nabla F(\z')|} 
 = \int_{\Sigma^\cF} w(\z) \frac{d\z\!\mid_{\Sigma^\cF}}{ |\nabla \cF(\z)|} =: \sigma_\infty(w; \cF), 
 \end{split}
 \]
  then we arrive at the following corollary from Theorem~\ref{th:3}:
  
  \begin{corollary}
  If $d\ge5$, the quadratic form $F$ is as in Theorem \ref{th:3},
  $\cF$ is a non-homogeneous quadratic form as above and $L$ is such that
  $
  \zz:= A^{-1} \z_* \in \Z^d_L, \  \tau L^2\in\Z,
  $
  then for any $0<\eps\le1$ 
   and $w \in \cC^{K_1,K_2}(\R^{d})\cap \cC^{0,K_3}(\R^{d})$ we have 
  $$
  \big| N_L(w; \cF) - \sigma_\infty(w; \cF)\, \sigma(A, L^2 m) L^{d-2} \big| \le C   L^{d/2+\eps}\left(\|w\|_{K_1,K_2}+\|w\|_{0,K_3}\right).
  $$
  Here the constants $K_1, K_2, K_3$  depend on $d$ and $\eps$, while $C$ depends on $d, \eps, A$ and $\tau, |\z_*|$. 
  \end{corollary} 


\medskip

{\bf Notation and agreements.} We write 
$
A  \lappr_{a,b}  B
$ 
if $A\le C B$, where the constant $C$ depends on $a$ and $b$. Similar, $O_{a,b}( \|w\|_{m_1, m_2}) $ stands for a quantity, bounded 
 in absolute value by 
$
C(a,b) \|w\|_{m_1, m_2}. 
$
 We do not indicate the dependence on the  matrix norms $\|A\|$, $\|A^{-1}\|$ and on the dimension $d$ since most of our estimates depend on 
 these quantities. 

We always  assume that  function $w$ belongs to the space $\cC^{m,n}(\R^d)$ with sufficiently large $m,n$. 
If in the statement of an assertion we employ the norm $\|w\|_{a,b}$ then we assume that $w\in\cC^{a,b}(\R^d)$. 

We denote $e_q(x) = e^{2\pi ix/q}$ and abbreviate $e_1(x)=:e(x)$. By $ \lceil \cdot \rceil$ we denote the ceiling function,
$
 \lceil x \rceil =\min_{n\in\Z}\{ n\ge x\}.
$
By $\N$ we denote the set of  positive integers.
\medskip

 {\bf Acknowledgements.} The authors thank Professor Heath-Brown for advising  them concerning the  paper \cite{HB}.  The work of AD was funded by a grant from the Russian Science Foundation (Project 20-41-09009) [Sections~1-7] and by the Grant of the President of the Russian Federation (Project MK-1999.2021.1.1) [Sections~A-C]. 
 Research of  SK was equally supported   by the Ministry of Science and Higher Education of the Russian Federation (megagrant No. 075-15-2022-1115) and  by  l'Agence Nationale de la Recherche  (France),  grant    17-CE40-0006.

\subsection{Scheme of the proof of Theorem \ref{th:3}}\label{sec:2}

Let $d\ge5$.  As it has been already discussed, if $w$ satisfies assumptions of the theorem with sufficiently large constants $K_i$  then $w$ is regular in the sense of Section~\ref{s_1.1}, so Theorem~\ref{th:2} applies. 
Then, according to \eqref{eq:sum} and \eqref{obvious}, 
\begin{equation}\label{eq:sum'}
N_L(w{ ; A, m})=c_L \,L^{-2}\sum_{\vc\in \Z^{d}}\sum_{q=1}^\infty q^{-d}
S_q(\vc) I_q(\vc)\,,
\end{equation}
where the sum $S_q(\vc)=S_q(\vc;A,L^2m)$ is given by \eqref{eq:S}  with $t={L^2m}$ and the integral 
 $I_q(\vc)$~-- by \eqref{I^0_q} with $\tilde w=w_L$, $Q=L$ and $t={L^2m}$,
\begin{equation}\label{eq:I'}
I_q(\vc; A,m,L) := \int_{\R^{d}} w\left(\frac{\z}{L}\right)h\left(\frac qL,
\frac{{ F^{{L^2m}}}(\z)}{L^2}\right) e_q(-\z\cdot \vc)
\,d\z\,.
\end{equation}
Denoting
$$
n(\vc;A,m,L)=\sum_{q=1}^\infty q^{-d}
S_q(\vc) I_q(\vc)\,,
$$ 
we have  $\ds{N_L(w{ ;A,m})=c_L L^{-2}\sum_{\vc\in\Z^{d}} n(\vc)}$. Then for 
 any $\gamma_1 \in (0,1/2)$ we  write $N_L$ as 
\begin{equation}\label{eq:N_L=}
N_L(w{ ;A, m})=c_L L^{-2}\big(J_0 + J_{<}^{\gamma_1} + J_{>}^{\gamma_1}\big),
\end{equation}
where 
\begin{equation}\label{eq:J><}
 J_0:=n(0)\,, \quad J_<^{\gamma_1}:=\sum_{\vc\ne 0,\,|\vc|\leq
   L^{\gamma_1}}  n(\vc) \,,\quad 
J_>^{\gamma_1}:=\sum_{|\vc|> L^{{\gamma_1}}}  n(\vc)\,.
\end{equation}
Proposition~\ref{l:19} (which is a modification of Lemmas~19 and  25  from \cite{HB})  implies that 
 $$
 |J_>^{\gamma_1}|\lappr_{\gamma_1,m} \|w\|_{N_0,2N_0+d+1}\,
 $$
with $N_0:=\lceil {d+(d+1)/{\gamma_1}} \rceil$ (see Corollary~\ref{c:J>}). In 
Proposition~\ref{l:J_<},  following  Lemmas~22 and   28  from \cite{HB}, we show   that 
\begin{equation}\label{eq:J_<}
 |J_<^{\gamma_1}|\lappr_{{\gamma_1}{ ,m}} L^{d/2+2+\gamma_1(d+1)} \left(\|w\|_{\bar
   N,d+5}+\|w\|_{0,\bar N+3d+4}\right) \,,
 \end{equation}
$\bar N= \lceil d^2/{\gamma_1}\rceil-2d$. 

To analyse $J_0$ we write it as  $J_0=J_0^++J_0^-$, where 
\be\lbl{eq:n(0)=}
J^+_0:=\sum_{q>\rho L}  q^{-d}
S_q(0) I_q(0)\,,\quad J_0^-:=\sum_{q\leq \rho L} q^{-d}
S_q(0) I_q(0)\,,
\ee
with $\rho = L^{-\gamma_2}$ for some  $0<\gamma_2<1$ to be determined.
 Lemma~\ref{l:I_A}, which is a combination of Lemmas 16 and 25 from
 \cite{HB}, modified using the results from Section~\ref{sec:7}, 
 implies that 
  $$
  \Bigl|J_0^+\Bigr|\lappr L^{d/2+2+\gamma_2(d/2-1)}|w|_{L_1}{\lappr}
  L^{d/2+2+\gamma_2(d/2-1)}\|w\|_{0,d+1}  .
  $$
 Finally  Lemma~\ref{l:I_B}, which is a combination of Lemma~13 and  simplified 
 Lemma~31  from \cite{HB} with the  results from Section~\ref{sec:7}, establishes  that $J_0^- $ equals
 $$
  L^{d} \sigma_\infty(w) \sigma(A,{L^2m}) 
 + O_{\gamma_2{ ,m}}\Big(
 \big(\|w\|_{d/2-2,d-1} +\|w\|_{0,d+1}\big) L^{d/2+2+\gamma_2{ (d/2-2)}}\Big)
 $$
 (see \eqref{lead_term} and \eqref{77}). 
 Identity \eqref{eq:N_L=} together with the estimates above  implies 
 the desired result if we choose $\gamma_2=\eps/{(d/2-1)}$ and
 $\gamma_1= \eps/(d+1)$.
  Uniform in $L$ and $m$ boundedness of  the product $\sigma(A,L^2m)$ follows from Lemma~\ref{l:31HB}.

\subsection{Scheme of the proof of Theorem \ref{th:3-4}}\label{sec:2-4}
 In this section we assume that $d=4$ and $m=0$. The proof proceeds exactly as in the previous section up to formula
\eqref{eq:J_<}, which is not sharp enough for the case $d=4$ and
should be replaced by
\be\lbl{this_bound1}
\left|J_<^{\gamma_1}- L^d\sum_{\vc\neq 0}\eta(\vc){\sigma^*_\vc(A)
\sigma_{\infty}^\vc}(w;A,L)\right|\lappr_{{\gamma_1}}
L^{7/2+(d+4)\gamma_1}\|w\|_{\tilde K_1,\tilde K_2} 
\ee
for appropriate constants $\tilde K_1, \tilde K_2$,
where the terms
 $\sigma^*_\vc(A)$ are introduced in \eqref{eq:sigma_p}, 
terms ${\sigma_{\infty}^\vc}(w;A)$ are given by
\be\lbl{sigma_infty^c}
{\sigma_{\infty}^\vc}(w;A,L): = L^{-d} \sum_{q=1}^\infty
q^{-1}I_q(\vc;A,0,L)  \,,
\ee
and the constants $\eta(\vc)=\pm 1$  are defined in Lemma~\ref{l:30}. In  particular, $\eta(0)=1$ if the determinant $\det A$ is a square of an integer and $\eta(0)=0$ otherwise.
The
proof of the bound \eqref{this_bound1} makes use of Lemma~\ref{l:30} (Lemma~30 of \cite{HB}),
involving only minor modifications of the argument in  \cite{HB} and is left to the reader.

The bound on $J_0$ must be refined too and this is done in Appendix~\ref{s:appendix}.
We consider only the case when the determinant $\det A$ is a square of an integer, so in particular $\eta(0)=1$. The opposite case can be obtained by minor modification of the latter, following \cite{HB} (see Appendix~\ref{s:appendix} for a discussion). In Proposition~\ref{l:n(0)-4}, which is a
combination of Lemmas~13, 16 and 31 of \cite{HB}, modified using
Proposition~\ref{l:I_q(0)=}, we prove that in the case of square determinant $\det A$
\begin{equation*}
  \begin{split}
 J_0= &\sigma_\infty(w)\sigma^*(A)L^d\log L + K(0)
 L^d \\
& + O_{\eps}\big(L^{d-\eps}
 \big(\|w\|_{d/2-2,d-1}+\|w\|_{0,d+1}\big)\big)  \,,
  \end{split}
\end{equation*}
        where a constant $K(0)=K(0;w,A)$ is defined in Section~\ref{sec:app1}.
 Again, identity \eqref{eq:N_L=} together with the estimates above  implies 
 the desired result if we choose $\gamma_1=(\tfrac12-\eps)/(d+4)$ and put
\be\lbl{def:sigma_1}
\sigma_1(w;A,L):=K(0)+ \sum_{\vc\neq 0}\eta(\vc) \sigma^*_\vc(A)
\sigma_{\infty}^\vc(w;A,L)  \,.
\ee
Finiteness of the products $\sigma^*_\vc(A)$ follow from
Lemma~\ref{l:31} while the claimed in the theorem estimate for the
constant  $\sigma_1(w;A,L)$ is established in
Section~\ref{s:sigma_1}.

 \section{ Series $S_q$ }  
 Now we start to prove Theorem \ref{th:3}, following the scheme presented in Section \ref{sec:2}.   Part of the  assertions, forming the proof, do not use that $d\ge5$. 
 So below 
 in all assertion involving the dimension $d$,  we indicate the real requirements on $d$.
  We recall that the constants in estimates 
 may depend on $d$ and $A$, but this dependence is not indicated (see {\it Notation and agreements}).

 In the present section we analyse the sums $S_q(\vc)= S_q(\vc; A, L^2m)$ entering, in particular,  the definitions of the  singular series 
 $\sigma(A,{L^2m})$ and $ \sigma_p(A,{L^2m})$.
 \begin{lemma}[Lemma 25 in \cite{HB}]\label{l:25HB}
 	For any $d\geq 1$
 	we have $|S_q(\vc;A,L^2m)|\lappr q^{d/2+1}$, uniformly in $\vc\in\Z^{d}$.
 \end{lemma}
\begin{proof} According to \eqref{eq:S}, an application of
  the Cauchy-Schwarz inequality shows that
  \begin{equation}\label{eq:S_q=}
  \begin{split}
&|S_q(\vc)|^2 \leq \phi(q) {\sum_{a(\!\!\!\!\!\!\mod
		q)}}^* \, \Big|\sum_{\vb(\!\!\!\!\!\!\mod q)} e_q(a{ F^{{L^2m}}}(\vb) + \vc\cdot \vb)   \Big|^2  \\
&= \phi(q) {\sum_{a(\!\!\!\!\!\!\mod q)}}^*  \sum_{\vu,\vv(\!\!\!\!\!\!\mod q)} 
e_q\big(a({ F^{{L^2m}}(\vu)-F^{{L^2m}}(\vv)}) + \vc\cdot (\vu-\vv)\big),
\end{split}
\end{equation}
where $\phi(q)$ is the Euler totient function. Since
$
F^t(\z) = \tfrac12 A\z\cdot\z -t, 
$
then 
$$
{ F^{{L^2m}}(\vu)-F^{{L^2m}}(\vv)}=(A\vv)\cdot \vw
+ F(\vw)=\vv\cdot A\vw + F(\vw). 
$$
So
$$
e_q\big(a({ F^{{L^2m}}(\vu)-F^{{L^2m}}(\vv)}) + \vc\cdot (\vu-\vv)\big)=
e_q\big(aF(\vw) + \vc\cdot \vw\big)\,
e_q(a\vv\cdot A\vw).
$$
Now we see that the summation over $\vv$ in \eqref{eq:S_q=} 
produces  a zero contribution, unless each component of the vector $A\vw$ is divisible by $q$. This property  holds
for at most a finite number $N$ of vectors $\vw$, where the constant $N$ depends only on $\det A$. 
Thus,  
$$|S_q(\vc)|^2 \lappr \phi(q) {\sum_{a(\!\!\!\!\!\!\mod
		q)}}^* \, \sum_{\vv(\!\!\!\!\!\!\mod q)} 1 \le
\phi^2(q)\, q^{d}. $$   
\end{proof}
	

The lemma's assertion  shows that  the sums $\sigma^{\vc}_p$, defined in \eqref{eq:sigma_p}, are finite:
\begin{corollary}\lbl{cor:sigma}
	If $d\ge 5$, for any prime $p$ we have   $\big|\sigma^{\vc}_p(A, L^2m)\big|\lappr  1$. 
\end{corollary}

Recall that  $\sigma(A,L^2m)=\prod_p \sigma_p(A,L^2m)$ (see  \eqref{77}). 

\begin{lemma}\label{l:31HB}
	For any $d\ge5$  and $1\le X \le\infty$
	 we have $$\sum_{q\leq X}q^{-d} S_q(0) =  \sigma(A,L^2m) +
        O(X^{-d/2+2}).$$
	In particular,  $\sigma(A,L^2m)=\sum_{q=1}^\infty q^{-d} S_q(0) $. So  $|\sigma(A,L^2m)|\lappr 1$ in view  Lemma~\ref{l:25HB}. 
\end{lemma}
\begin{proof}
We start by showing the multiplicative property of trigonometric sums
\be\lbl{Smult}
S_{qq'}(0) = S_q(0) S_{q'}(0)\,,
\ee
whenever $(q,q')=1$ (cf. Lemma { 23} from \cite{HB}).
By definition
$$
S_{qq'}(0) = {\sum_{ a(\!\!\!\!\!\!\mod qq')}}^* \,
\sum_{\vv(\!\!\!\!\!\!\mod qq')} e_{qq'}(aF^{{L^2m}}(\vv) )\,.
$$
When $(q,q') = 1$ we can replace the summation on $a$ ${(\!\!\!\! \mod qq')}$
 by a double
summation on $a_q$ modulo $q$ and $a_{q'}$ modulo $q'$ by writing $a=q
a_{q'}+q'a_q$. Thus 
$$
S_{qq'}(0) ={\sum_{ a_q(\!\!\!\!\!\!\mod q)}}^* \, {\sum_{ a_{q'}(\!\!\!\!\!\!\mod q')}}^* 
\sum_{\vv(\!\!\!\!\!\!\mod qq')} e_{q}(a_qF^{{L^2m}}(\vv) )
e_{q'}(a_{q'}F^{{L^2m}}(\vv) )  \,.
$$
Then we replace the summation on $\vv$ ${(\!\!\!\! \mod qq')}$ with the double summation
on $\vv_q$ modulo $q$ and $\vv_{q'}$ modulo $q'$ by writing $\vv=
q\bar q \vv_{q'} + q'\bar q' \vv_q$, where $\bar q$ and $\bar q'$ are
defined through relations $q\bar q=1\,(\!\!\!\mod q')$ and $q'\bar q' = 1\,(\!\!\!\mod q)$. We observe that
$$
F^{{L^2m}}(\vv) = q^2\bar q^2F(\vv_{q'}) + q'^2\bar q'^2F(\vv_{q}) +
q\bar qq'\bar q' A\vv_{q'}\cdot \vv_q -{L^2m} \,,
$$
so that
$$
e_q(a_qF^{{L^2m}}(\vv))= e_q(a_qq'^2\bar q'^2 F(\vv_q) -a_q{L^2m}) =
e_q(a_q F^{{L^2m}}(\vv_q)),
$$
by the definition of $\bar q'$ and since $e_q(q N)=1$ for any integer $N$. Similar, 
$$
e_{q'}(a_{q'}F^{{L^2m}}(\vv))=
e_{q'}(a_{q'} F^{{L^2m}}(\vv_{q'}))  \,.
$$
This gives \eqref{Smult}.

Next we note that, due to Lemma~\ref{l:25HB}, 
\be\lbl{sum_q}
\sum_{q\geq X}q^{-d} |S_q(0)|\lappr \sum_{q\geq X}q^{-d/2+1} \lappr X^{-d/2+2}.
\ee
By \eqref{Smult} and the definition of  $\sigma$, 
$$
\sigma=\lim_{n\to\infty} \sigma^n, \qquad \sigma^n=
\prod_{p\leq n}\sum_{l=0}^n p^{-dl}S_{p^l}(0)=
\sum_{q\in P_{n}} q^{-d}S_q(0),
$$
where $p$ are primes and $P_{n}$ denotes the set of natural numbers $q$ with prime factorization of the form  $q=p_1^{k_1}\cdots p_m^{k_m}$, where $2\le p_1<p_2\dots <p_m\le n$,  $k_j\leq n$ and $m\ge0$ ($m=0$ corresponds to $q=1$). Since any $q\le n$
belongs to $P_n$, then according to \eqref{sum_q}, 
$$
\Big|\sum_{q\in P_{N}} q^{-d}S_q(0) - \sum_{q\leq  X} q^{-d}S_q(0)\Big|  \lappr X^{-d/2 +2} \quad \forall\, N\ge X,
$$
for any finite $X>0$. Passing in this estimate to a limit as $N\to\infty$ we recover the assertion if $X<\infty$.  Then the 
result with $X=\infty$ follows in an obvious way. 
\end{proof}

\section{Singular integrals  $I^0_q$}
\label{sec:3}

\subsection{Properties of $h(x,y)$}
Following \cite{HB}, Section~3, we construct the function $h(x,y)\in C^\infty(\R_>,\R)$,  entering Theorem~\ref{th:1},
starting from the weight function
$w_0\in C_0^\infty(\R)$, defined as
\begin{equation}\label{eq:or_weight}
w_0(x) = \left\{\begin{array}{cc}
\exp\left(\frac1{x^2-1}\right) & \mbox{for }|x|<1\\
0 & \mbox{for }|x|\ge 1
\end{array}
\right. \,.
\end{equation}
We denote 
$
c_0:= \int_{-\infty}^{\infty}w_0(x)\,dx 
$
and introduce the shifted weight function
\begin{equation*}
  \omega(x)=\tfrac{4}{c_0}w_0(4x-3)\,,
\end{equation*}
which  of course belongs  to $C^\infty_0(\R)$. Obviously, 
 $ 0\le \omega\le 4e^{-1}/c_0$, $\omega$ is supported on $(1/2,1)$, and
$
\int_{-\infty}^{\infty}\omega(x)\,dx =1 \,.
$

The required function $h: \R_{>0}\times\R\to\R$ is defined in terms of $\omega$ as $ h(x,y) := h_1(x)-h_2(x,y)$ with 
\begin{equation}\label{eq:h}
  \begin{split}
h_1(x):={ \sum_{j=1}^\infty}\frac{1}{xj}\omega(xj)\,,\quad h_2(x,y):= { \sum_{j=1}^\infty}\frac{1}{xj}\omega \left(\frac{|y|}{x
  j}\right)  \,.
\end{split}
\end{equation}
For any fixed pair $(x,y)$ each of the two 
sum in $j$ contains a finite number of nonzero terms.  So $h$ is a smooth function. 

In \cite{HB}, Section~3, it is shown how to derive Theorem~\ref{th:1} from
the definition \eqref{eq:h}.\footnote{Actually it is proved there that any 
  function $h$ defined through \eqref{eq:h} with arbitrary weight function $\omega\in
  C_0^\infty(\R)$, supported on $[1/2,1]$, may provide a
  representation of $\delta(n)$. } Here we limit ourselves
 to providing some relevant properties of $h$,  
proved in Section~4 of \cite{HB}. In particular these properties imply that for small $x$,
  $h(x,y)$ behaves as the Dirac delta function in $y$

\begin{lemma}[Lemma 4 in \cite{HB}]\label{l:4}
  We have:
  \begin{enumerate}
\item    $h(x,y)=0$ if $x\ge 1$ and $|y|\le x/2$.
\item If $x\le 1$ and $|y|\le x/2$, then  $h(x,y)=h_1(x)$, and for any $m\geq 0$
  $$ 
  \Big|\frac{\partial^m h(x,y)}{\partial x^m} \Big| \lappr_m \frac{1}{x^{m+1}}\,.
  $$
  \item If $|y|\ge x/2$, then for any $m,n\geq 0$
  $$
   \Big|\frac{\partial^{m+n} h(x,y)}{\partial x^m\partial y^n} \Big| \lappr_{m,n}
  \frac{1}{x^{m+1}|y|^n}\,.
  $$
  \end{enumerate}
  \end{lemma}

\begin{lemma}[Lemma 5 in \cite{HB}]\label{l:5}
  Let $m,n,N\ge 0$. Then for any $x,y$
  $$
   \Big|\frac{\partial^{m+n} h(x,y)}{\partial x^m\partial y^n}  \Big|\lappr_{N,m,n}
  \frac{1}{x^{1+m+n}}\left(\delta(n)x^N + \min\left\{1,\left(x/|y|\right)^N\right\} \right)\,.
  $$
  \end{lemma}

Lemma~\ref{l:5}  with $m=n=N=0$ immediately implies
\begin{corollary}\lbl{c:h}
	For any $x,y\in\R_>\times\R$ we have $|h(x,y)|\lappr 1/x$.
\end{corollary}

\begin{lemma}[Lemma 6 in \cite{HB}]\label{l:6}
  Fix $ X\in \R_{>0}$ and  $0<x<C\min\left\{1,X\right\}$, for some $C>0$. Then for any $N\ge0$, 
  $$
 \int_{-X}^X h(x,y)\,dy = 1 + O_{N,C}\left(Xx^{N-1}\right) +
 O_{N,C}\left(\frac{x^N}{X^N} \right).
  $$
  \end{lemma}

\begin{lemma}[Lemma 8 in \cite{HB}]\label{l:8}
  Fix $X\in \R_{>0}$ and $n\in \N$. Let $x<C\min\left\{1, X\right\} $ for
  $C>0$. Then 
  $$
 \Big|\int_{-X}^X y^nh(x,y)\,dy  \Big|\lappr_{N,C} X^n\left(Xx^{N-1} +
\frac{x^N}{X^N} \right)\,.
  $$
\end{lemma}

The previous results are used to prove the key Lemma~9 of \cite{HB},
which can be extended to the following
\begin{lemma}\label{l:9} Let a function  $f\in \cC^{M-1,0}(\R)\cap L^1(\R)$, $M\geq 1$, be such that its $(M-1)$-st derivative 
  $f^{(M-1)}$ is absolutely continuous on $[-1,1]$, and let  $0<x\le C$ for some   $C>0$. 
  Then for any $0<\beta\le 1$ and any $N\ge 0$, 
    \begin{equation} \lbl{l9HB}
      \begin{split}
\int_\R f(y) h(x,y) \, dy  = & f(0) + O_{M}\left(\frac{x^{M}}{\beta^{M+1}}
\frac{1}{X}\int_{-X}^{X}|f^{(M)}(y)|\,dy\right)\\
&+O_{N,C}\left((x^N+\beta^N)\left(\|f\|_{M-1,0}+x^{-1} |f|_{L_1}\right)\right)\,,
      \end{split}
  \end{equation}
  where $X:=\min\left\{1,x/\beta\right\}$.
\end{lemma}
    {\it Proof.} By
    Lemma~\ref{l:5} with $m=n=0$,  for any $N\geq 0$ we have
    $|h(x,y)| \lappr_N (x^N+\beta^N)x^{-1}$ if $|y|\ge X$. So  the tail-integral for $\int fh \,dy$ 
    may be     bounded as
    \begin{equation}\label{eq:l9-1}
    \left|\int_{|y|\ge X} f(y) h(x,y) \, dy\right|\lappr_N (x^N+\beta^N)x^{-1} \int_{|y|\ge X} |f(y)|\,dy\lappr_N
    (x^N+\beta^N)x^{-1} |f|_{L_1}\,.
    \end{equation}

    For the integral in $|y|<X$, instead we take the Taylor
    expansion of $f(y)$ around zero and get that 
    \begin{equation}\label{eq:l9-2}
    \begin{split}
        \int_{-X}^X f(y) h(x,y) \, dy  
        = &\sum_{j=0}^{M-1}
        \frac{f^{(j) }(0)}{j!}\int_{-X}^X y^j h(x,y)\, dy \\&+
        O_M\left(\frac{X^{M}}{x} \int_{-X}^{X}|f^{(M)}(y)|\,dy\right)\,,
        \end{split}
    \end{equation}
    by Corollary~\ref{c:h}. Next we use Lemma~\ref{l:6} with $N$ replaced by $N+1$     to
    get that 
    \begin{equation}\label{eq:l9-3}
        f(0) \int_{-X}^X  h(x,y)\, dy = f(0) + 
         O_{N,C}\left(\|f\|_{0,0}\left(Xx^{N} 
        +\frac{x^{N+1}}{X^{N+1}}\right)\right),
    \end{equation}
    while by Lemma~\ref{l:8}, for any $j>0$ we have
    \begin{equation}\label{eq:l9-4}
\left|      \frac{f^{(j) }(0)}{j!}\int_{-X}^X y^j h(x,y)\, dy\right|
   \lappr_{N,j,C}   \|f\|_{j,0}X^j\left(Xx^{N} 
   +\frac{x^{N+1}}{X^{N+1}}\right)\,.
    \end{equation}
    Putting together \eqref{eq:l9-1}--\eqref{eq:l9-4}, we obtain the desired estimate. 
    Indeed, since $X\leq x/\beta$, then the term $O_M$ in \eqref{eq:l9-2} is bounded by that in \eqref{l9HB}.
    Moreover, as
     $(x/X)^{N+1}=\max(x^{N+1},\beta^{N+1})\lappr_{C} Cx^N +\beta^N$, then
      the brackets in \eqref{eq:l9-3} and \eqref{eq:l9-4} are $\lappr_{C} x^N + \beta^N$, where we also used that $X\leq 1$. 
    \qed

Lemma~\ref{l:9} is needed for the proof of Theorem~\ref{th:3-4},  while for Theorem~\ref{th:3} we only need  its simplified version:

\begin{corollary}\label{c:9} Let an
  integrable function $f$ belong to the class  $\cC^{{M},0}(\R)$, ${M}\in \N$, and  $0<x\le C$ for some 
  $C>0$. Then, for any $0<\de <1$, 
  \begin{equation*}
\int_\R f(y) h(x,y) \, dy = f(0) + O_{{M},C,{\de}}\left(x^{{M}-{\de}}
\left(\|f\|_{{M},0}+|f|_{L_1}\right)\right)\,.
  \end{equation*}
\end{corollary}
{\it Proof.} 
The assertion follows from Lemma~\ref{l:9}   by choosing 
	for any $0<\delta<1$, 
	$\beta=x^{\de/(M+1)}$ if $x\le 1$ and $\beta=1$ if $x>1$.  Indeed, then for $ 0<x\le1$ we have that 
	$
	x^M \beta^{-(M+1)} = x^{M-\delta},
	$
	and that 
	$$
	(x^N+\beta^N) x^{-1} \le 2\beta^N x^{-1}   \le 2 x^{M-\delta} 
	\qquad\text{if} \;\; N\ge N_\delta =(M-\delta +1)(M+1) /\delta.
	$$
	While if $1\le x\le C$, then $x^M \le C^\delta x^{M-\delta}$, and choosing 
	  $N=0$ we get that $(x^N+1) =2 \le 2x^{M-\delta}$. The obtained relations imply the assertion. 
	   \qed

    \subsection{Approximation for $I_q(0)$}
   In what follows it is convenient to write the integrals $I_q(\vc;
   A, L^2m)$ as  
    \be\lbl{Iq-t_Iq}
    I_q(\vc)=L^{d} \tilde I_q(\vc),
    \ee
    where
    \be\lbl{tilde_I_q}
    \tilde I_q (\vc) = \tilde I_q (\vc; A,m,L) =
     \int_{\R^{d}} w(\z)\,
    h\left(\frac qL, { 
    	F^m}(\z)\right) e_q(-\z\cdot \vc L)
    \,d\z\,.
    \ee    
 The proposition below replaces Lemmas~11, 13 and
Theorem~3 of \cite{HB}. In difference with those results we  do not
assume  that  $0\notin\supp w$.
Since for $\vc = 0$ the exponent $e_q$ in the definition of the integral $I_q(\vc)$ equals one,  we can consider $I_q(0)$ as a function of a {\it real} argument $q\in\R$, and we do so in the proposition below; we will use this in Appendix~\ref{s:appendix}.

    \begin{proposition}\lbl{l:I_q(0)=} 
     Let  $q\in\R$, $q\le C L$ with some $C> 0$. 
   
   a)  If $d\ge5$ and    ${\N}\ni {M}< d/2-1$,  then    for any
       $\de>0$,
    \be\lbl{I_q(0)=}
    \begin{split}
      I_q(0;A, m, L) &= L^{d} \sigma_\infty(w;A,m)  \\    
      &+
      O_{m, {M},C,{\de}}\left(q^{{M}-\delta} L^{d-{M}+{\de}}
     { \| w\|_{{M},d+1}}\right) .
     \end{split}
      \ee
      
      b)  
       If $d = 4$, $\N \ni {M}\le d/2-1$
        and $m=0$, then for any $0<\beta\le 1$ and $N\ge 0$,
      \begin{equation}\lbl{I_q(0)=-4}
        \begin{split}
      I_q(0; A,0,L) =& L^{d} \sigma_\infty(w; A,0)
     +  O\left(\beta^{-{M}-1}q^{M}L^{d-{M}}\Big\lan\log\big(\frac{q}{L\beta}\big)\Big\ran { \|
     w\|_{{M},d+1}}\right)\\
      &+ O_{C,N}\left((q^NL^{d-N}+\beta^N)({ \|
      w\|_{{M}-1,d+1}} +Lq^{-1}\|
      w\|_{0,d+1})\right) .
      \end{split}
      \end{equation}     
    \end{proposition}
{\it Proof.} For $d\ge{4}$, applying the  co-area formula 
(see [3], Theorem 6.3) we re-write the integral in (3.9) with $c=0$
in terms of integrals over hyper-surfaces $\Sigma_t$ as follows:
\be\label{inte}
\tilde I_q(0)=\int_\R \II({ m+}t) h(q/L,t)\, dt\,, \qquad \II(t) = \int_{\Sigma_t} w(\z)\,\mu^{\Sigma_t}(d\z)\,,
\ee 
where the  measure $\mu^{\Sigma_t}$ is the same as in \eqref{lead_term}. 
By  Theorem~\ref{c_integrals},
\be\label{yes}
\| \II\|_{k, \tilde K}  \lappr_{k, K, \tilde K}  \| w\|_{k,K} \;\; \;\text{if \ $\tilde K < \frac{K+2-d}2$, \quad { $K>d$}, }
\ee
and $ k< d/2-1$.
Denote $f^m(y) = \II(m+y)$. Then
$
\| f^m\|_{k,\tilde K} \lappr_{m, \tilde K} \| \II \|_{k,\tilde K}, 
$
and by \eqref{yes} 
\be\label{yyes}
| f^m|_{L_1} = | \II |_{L_1}  \lappr  \| \II \|_{0, 4/3}  \lappr  \| w\|_{0, d+1}.
\ee
To prove a) we apply  Corollary~\ref{c:9} with  $f=f^m$ and 
$x=q/L$   to the first integral in \eqref{inte}. 
Note that 
	$
	f^m(0) = \II(m) = \sigma_\infty(w; A,m).
	$
Then, using \eqref{yes} with $\tilde K=0, \, { K=d+1}$ and $k= M$
jointly with \eqref{yyes} 
 we get that 
$$
\tilde I_q(0) = \sigma_\infty(w) +O_{M, m,C, \delta} \big( q^{M-\delta} L^{-M+\delta} {  \| w\|_{M, d+1}}\big). 
$$
So \eqref{I_q(0)=} follows. 
\medskip

To establish  \eqref{I_q(0)=-4}, we apply Lemma~\ref{l:9} to write the integral in \eqref{inte} with $m=0$ as 
\begin{equation*}
  \begin{split}
\int_\R \II(t) h(x,t)\,dt &= \II(0) +
      O_{M}\left(\beta^{-M-1}x^{M} 
  \left(\frac1X\int_{-X}^{X}|\II^{(M)} (t)|\,dt\right)\right)\\
  &+O_{C,N}\left((x^N+\beta^N)(\|\II\|_{M-1,0}
+x^{-1} |\II |_{L_1})\right),
  \end{split}
\end{equation*}
where $x=q/L$ and $ X=\min\{1,x/\beta\}.$
By applying  Theorem~\ref{c_integrals}, with $k=M$
  and $M=d+1$, we get
$$
\int_{-X}^{X}|\II^{(M)}(t)|\,dt \, \lappr
X\lan \log X\ran {  \|w\|_{M,d+1}}\,.
$$
Using this estimate jointly with \eqref{yes} and \eqref{yyes} we arrive at \eqref{I_q(0)=-4}.

\section{The $J_0$ term }\label{sec:4}

In this section we prove the following proposition concerning  the
term $J_0$ defined in \eqref{eq:J><}:

\begin{proposition}\lbl{l:n(0)}
  { Let $d\ge5$}. Then for any $0<\gamma_2<1$,
	$$
	\big|J_0-L^{d}\sigma_\infty(w){  \sigma(A,L^2m)}\big|
        \lappr_{\gamma_2{ ,m}}
	 L^{\frac{d}2+2 +\gamma_2(\frac{d}2-1)}
        {  \|w\|_{\lceil d/2\rceil-2,d+1}}.
	$$
\end{proposition}
{\it Proof.}
To establish the result we write $J_0$ in the form
\eqref{eq:n(0)=}. 
Then the assertion   follows from Lemmas~\ref{l:I_A} and
\ref{l:I_B} below which estimate the terms $J_0^+$ and $J_0^-$, noting  that $|w|_{L_1}\lappr
\|w\|_{0,d+1}$.
\qed

\begin{lemma}\lbl{l:I_A}
	Assume that $w\in L_1(\R^{d})$ and $d\ge3$. Then we have the
        bound $|J_0^+|\lappr
        L^{d/2+2 +\gamma_2(d/2 -1)}|w|_{L_1}$, {  for any $\gamma_2\in(0,1)$.}
\end{lemma}
{\it Proof.}
Since according to Lemma~\ref{l:25HB} $|S_q(0)|\lappr q^{d/2+1}$, then 
$$
|J_0^+|\lappr\sum_{q>L^{1-\gamma_2}} q^{-d/2+1}I_q(0).
$$
 Writing  integral $I_q$ as in  \eqref{Iq-t_Iq}, by Corollary~\ref{c:h} we get
$
\ds{|I_q(0)|\lappr \fr{L^{d+1}}{q}  |w|_{L_1}.} 
$
Therefore,
\begin{equation*}
  \begin{split}
|J_0^+|&\lappr L^{d+1}|w|_{L_1} \sum_{q>L^{1-{\gamma_2}}} q^{-d/2 }
\lappr  L^{d+1}|w|_{L_1}L^{(-d/2+1)(1-\gamma_2)}\\
&=
L^{d/2+2 +\gamma_2(d/2-1)}|w|_{L_1}.
  \end{split}
\end{equation*}
\qed

\smallskip

\begin{lemma}\lbl{l:I_B} 
	{  Let $d\ge5$. Then for any $\gamma_2\in(0,1)$},
	$$
	J_0^-=L^{d}\sigma_\infty(w){  \sigma(A,L^2m)} + O_{\gamma_2{ ,m}}\big(  L^{d/2+2 +\gamma_2{ (d/2-2)}} {  \|w\|_{\lceil d/2\rceil-2,d+1}} \big).
	$$
\end{lemma}
{\it Proof.}
Inserting \eqref{I_q(0)=}  with $C=1$  into the definition of the term $J_0^-$, we get 
$
J_0^-=I_A+I_B,
$
where
\begin{align}\non
&I_A :=L^{d}\sigma_\infty(w)\sum_{q\leq L^{1-\gamma_2}} q^{-d}S_q(0), 
\\\non 
&|I_B|\lappr_{M,\delta{ ,m}}L^{d-M+\delta}{ \|w\|_{M,d+1}}
\sum_{q\leq 
  L^{1-{\gamma_2}}}S_q(0) q^{-d+M} \,,
\end{align}
for $M{ < d/2-1}$ and any $\delta>0$. \ 
Lemma~\ref{l:31HB} implies that $$\sum_{q\leq L^{1-\gamma_2}} q^{-d}S_q(0)={  \sigma(A,L^2m)} + O(L^{(-d/2+2)(1-\gamma_2)}),$$ so
$$
I_A=L^{d}\sigma_\infty(w){  \sigma(A,L^2m)} + O(\sigma_\infty(w)
L^{d/2+2 +\gamma_2(d/2-2)})\,,
$$
whereas $|\sigma_\infty(w)|=|\II({ m})|\le{ \|\II\|_{0,0}\le
\|w\|_{0,d+1}}$  on account of ~\eqref{yes}.  As for the term
$I_B$, Lemma~\ref{l:25HB} implies that
$$
|I_B|\lappr_{M,\delta{ ,m}}
L^{d-M+\delta}{ \|w\|_{M,d+1}}\sum_{q\leq 
  L^{1-\gamma_2}}q^{-d/2+1+M}\,. 
$$
Choosing $M={ \lceil d/2\rceil} -2$ and $\delta =\gamma_2/2$, we get 
$$
|I_B|\lappr_{\delta{ ,m}} {  \|w\|_{\lceil d/2\rceil-2,d+1}}
L^{d/2+2+\delta}\ln L
\lappr_{\gamma_2{ ,m}}
{  \|w\|_{\lceil d/2\rceil-2,d+1}} \, L^{d/2+2 +\gamma_2}\,.
$$
\qed

\section{The  $J_>^{\gamma_1}$ term}\label{sec:5}

We provide here an estimate of the term $J_>^{\gamma_1}$ defined in
\eqref{eq:J><}.  The key point of the proof 
 is an adaptation of Lemma~19 of
\cite{HB} to our case. We recall the notation~\eqref{Iq-t_Iq}. 

\begin{proposition}\label{l:19}
   { For any $d\geq 1$}, $ N>0$ and { $\vc\ne 0$},
\be\lbl{eq:19}
|\tilde I_q(\vc)|\lappr_{N{ ,m}} \frac Lq |\vc|^{-N}\left\|w\right\|_{N,2N+d+1}
  \ee
\end{proposition}

{\it Proof.}
Let $f_q(\z):= w\left(\z\right)
  h\left(\frac qL, 
 {  F^m}(\z)\right)$. Since
  $$
  \frac{i}{2\pi}\frac{q}{L} |\vc|^{-2}\left( \vc\cdot \nabla_{\z}\right)
  e_q(-\z\cdot \vc L)  =  e_q(-\z\cdot \vc L)\,,
  $$
then integrating by parts $N$ times the integral  \eqref{tilde_I_q} we get that

\begin{equation*}
  \begin{split}
    \left| \tilde I_q(\vc)\right| &\le \left(\frac{q}{2\pi L} |\vc|^{-2}\right)^N \int_{\R^{d}} \left|\left( \vc\cdot
    \nabla_{\z}\right)^N f_q(\z)\right|\,d\z\\
    &\lappr_{N} \left(\frac qL\right)^N |\vc|^{-N}
    \sum_{0\le n\le N}\int_{\R^{d}}\max_{0\le l\le n/2}
    \left|\frac{\partial^{n-l}}{\partial y^{n-l}} h\left(\frac
    qL, {  F^m}(\z)\right) \right|\\
        &\qquad\qquad\times|\z|^{n-2l}
    \left|
    \nabla_{\z}^{N-n} w(\z)\right|\,d\z
\,,
  \end{split}
\end{equation*}
where $\ds{\frac{\partial}{\partial y} h}$ stands for the derivative of  $h$ with respect to 
 the second argument. 

Assume first that $q\le L$. Then, by Lemma~\ref{l:5} with $N=0$,
\begin{equation*}
  \begin{split}
\max_{0\le l\le n/2}\left|\frac{\partial^{n-l}}{\partial y^{n-l}} h\left(\frac
    qL, {  F^m}(\z)\right) \right| |\z|^{n-2l}
    \left|    \nabla_{\z}^{N-n} w(\z)\right|\le\\ \left(
    L/q\right)^{n+1}\lan \z\ran^{-d-1} \|w\|_{N-n,n+d+1}\,.
  \end{split}
\end{equation*}
This implies \eqref{eq:19}  since $n\le N$. 
  Let now $q>L$. Then, due to item 1 of
Lemma~\ref{l:4}, $h$ is different from zero only if
\be\lbl{h ne 0}
2| {  F^m}(\z)| >\frac qL .
\ee  
Then for such $\z$ and for $l\le n$, item 3 of Lemma~\ref{l:4} implies that

$$
\left|\frac{\partial^{n-l}}{\partial y^{n-l}} h\left(\frac qL,  {  F^m}(\z)\right)\right| \lappr_{n-l} \frac Lq\frac{1}{| {  F^m}(\z)|^{n-l}}\lappr_{n-l}\Big(\frac{L}{q}\Big)^{n-l+1}.
$$
So 
\begin{equation*}
\begin{split}
\max_{0\le l\le n/2} \left|\frac{\partial^{n-l}}{\partial y^{n-l}} h\left(\frac
qL, {  F^m}(\z)\right) \right| |\z|^{n-2l}
\left|
\nabla_{\z}^{N-n} w(\z)\right|\lesssim\\ \max_{0\le l\le n}\frac{\left(
	L/q\right)^{n-l+1}}{\lan \z\ran^{2(N-n+l)}}
\frac{\|w\|_{N-n,2N-n+d+1}}{\lan\z\ran^{d+1}}\,.
\end{split}
\end{equation*}
Since from  \eqref{h ne 0} we have that $q/L \lappr_{{m}}
\lan\z\ran^2$,  
then the first fraction above is bounded by $(L/q)^{N+1}$,  and again 
 \eqref{eq:19} follows. 
\qed

As a corollary we get an estimate  for $J_>^{\gamma_1}$:
\begin{corollary}\label{c:J>}
  For $J_>^{\gamma_1}$ defined in \eqref{eq:J><} {  with $\gamma_1\in(0,1)$} and $d\ge3$ we have
  $$
|J_>^{\gamma_1}|\lappr_{{\gamma_1}{ , m}}\|w\|_{N_0,2N_0+d+1}\,,
$$
where $N_0:=\lceil {  d+(d+1)/{\gamma_1}}\rceil$.
\end{corollary}

{\it Proof.}
  Denoting by $|\cdot |_1$ the $l^1$-norm, by the definition of $J_>^{\gamma_1}$ we have
\begin{equation*}
  \begin{split}
|J_>^{\gamma_1}|&\lappr \sum_{s\ge L^{\gamma_1}} s^{d-1}\sum_{q=1}^\infty
q^{-d}\sup_{|\vc|_1=s}\ |S_q(\vc)| |I_q(\vc)|\\
&\lappr \sum_{s\ge L^{\gamma_1}} s^{d-1}\sum_{q=1}^\infty
q^{1-d/2}L^d\sup_{|\vc |_1=s} |\tilde I_q(\vc)|\\
&\lappr_{N{ , m}} \sum_{s\ge L^{\gamma_1}} s^{d-1}\sum_{q=1}^\infty
q^{-d/2} s^{-N} L^{d+1}\|w\|_{N,2N+d+1}\,,
  \end{split}
\end{equation*}
  where the second line follows from  Lemma~\ref{l:25HB}, while the
  third one -- from  Proposition~\ref{l:19}. The sum in $q$ is bounded by a constant. Choosing $N=N_0$ we  get that
  $$
  L^{d+1}\sum_{s\ge L^{\gamma_1}} s^{d-1} s^{-N} \le
  L^{d+1}\sum_{s\ge L^{\gamma_1}} s^{-1 - (d+1)/{\gamma_1}} \lappr 1 \,.
  $$
   This concludes the proof.
  \qed

 \section{ The $J^{\gamma_1}_<$ term }\label{sec:6}
 \subsection{The estimate}\label{subsec:6.1}
 Our next (and final) goal is to estimate the term $J^{\gamma_1}_<$ from \eqref{eq:N_L=}.
 
 \begin{proposition}\lbl{l:J_<}  For any $d\ge3$ and $\gamma_1\in(0,1/2)$,
 	$$|J_<^{\gamma_1}|\lappr_{{\gamma_1}{ ,m}}  L^{d/2+2+{\gamma_1}(d+1)} \left(\|w\|_{\bar
 		N,d+5}+\|w\|_{0,\bar N + 3d+4}\right) \,,$$
 	where
 	$\bar N=\bar N(d,{\gamma_1}):= \lceil d^2/{\gamma_1}\rceil-2d$. 
 \end{proposition}
 Proposition~\ref{l:J_<} will follow from the next lemma   which is a
 modification of Lemma~22 in \cite{HB} and is proved in the next subsection: 
 \begin{lemma}\lbl{l:22}
 For {  any $d\geq 3$} and 	$\vc\ne 0$,
 	$$|I_q(\vc)|\lappr_{{\gamma_1}{ ,m}}
 	L^{d/2+1+{\gamma_1}} 
 		\big (q / |\vc|\big)^{d/2-1-\ga_1}
 	 \left(\|w\|_{\bar N,d+5}+\|w\|_{0,\bar N
 		+3d+4}\right)\,,
 	$$
 	where $\bar N$ and $\gamma_1$ are the same as above.
 \end{lemma}
 
 {\it Proof of Proposition \ref{l:J_<}}.
 Accordingly to Lemma~\ref{l:25HB},
 \begin{align*}
 |J^{\gamma_1}_<| &\lappr \!
 \sum_{\vc\ne 0,\,|\vc|\leq
 	L^{{\gamma_1}}} \sum_{q=1}^\infty q^{-d}
 q^{d/2+1} |I_q(\vc)|\lappr L^{d{\gamma_1}}\!\! \max_{\vc\ne0:\,|\vc|\leq L^{\gamma_1}} |I_q(\vc)|\,
 \sum_{q=1}^\infty q^{-d/2+1}
 \\
 &=L^{d{\gamma_1}} \big(\sum_{q<L} + \sum_{q\geq L}\big) q^{-d/2+1}
 \max_{\vc\ne0:\,|\vc|\leq L^{\gamma_1}} |I_q(\vc)|=J_{-} + J_{+}\ ,
 \end{align*}
     with
\begin{align*}
J_- :=L^{d{\gamma_1}} \sum_{q<L} q^{-d/2+1}
\max_{\vc\ne0:\,|\vc|\leq L^{\gamma_1}}
|I_q(\vc)|\ ,\\
J_+:=L^{d{\gamma_1}} \sum_{q\geq L} q^{-d/2+1}
\max_{\vc\ne0:\,|\vc|\leq L^{\gamma_1}}
|I_q(\vc)|\ .
\end{align*}
 Corollary~\ref{c:h} together with \eqref{Iq-t_Iq}, \eqref{tilde_I_q} implies
 \be\lbl{Iq_large_q}
 |I_q(\vc)|\lappr \frac{ L^{d+1}}q |w|_{L_1}\,,
 \ee
 so that
 $$
 J_{+}\lappr L^{d{\gamma_1}}  L^{d+1} |w|_{L_1} \sum_{q\geq L} q^{-d/2}  \lappr  L^{d{\gamma_1} + d/2+2} |w|_{L_1}\lappr L^{d{\gamma_1} + d/2+2} \|w\|_{0,d+1}.
 $$
On the other hand,  since $|\vc|\geq 1$, from  Lemma~\ref{l:22} we get
 \begin{equation*}
 \begin{split}
 J_{-}
 &\lappr_{{\gamma_1}{ ,m}}  L^{d{\gamma_1}}  L^{d/2+1+{\gamma_1}}\left(\|w\|_{\bar N,d+5}+\|w\|_{0,\bar N
 	+3d+4}\right)\sum_{q< L} q^{-\ga_1}
 \\
 &\le   \left(\|w\|_{\bar N,d+5}+\|w\|_{0,\bar N
 	+3d+4}\right) L^{{\gamma_1}(d+1)+d/2+2}\,.
 \end{split}
 \end{equation*}
 \qed

\subsection{Proof of Lemma~\ref{l:22}}
We begin with
\subsubsection{Application of the inverse Fourier transform}
  Note that the proof is nontrivial only for $q\lappr L|\vc|$: indeed,   for
 any $\alpha>0$ the bound \eqref{Iq_large_q} implies
 that 
 $$
 |I_q(\vc)| \lappr_\alpha L^d|w|_{L_1} \lappr_{\alpha} L^d \big(L|\vc| / q\big)^{-d/2+1+\ga_1} |w|_{L_1}\, \quad \text{if} \quad q\ge \alpha L |\vc|,
 $$
 since $|\vc|\ge 1$ and $-d/2+1+\ga_1<0$. So, it remains to use again   the inequality $|w|_{L_1}\lappr \|w\|_{0,d+1}.$
 
 Let us take a small enough $\alpha=\alpha(d,\gamma_1, A) \in(0,1)$   and assume that $q< \alpha L|\vc|$. Consider the 
 function $w_2(x)=1/(1+x^2)$ and set 
 \be\lbl{tilde_w_def}
 \tilde w(\z):= \frac{w(\z)}{w_2( {  F^m}(\z))} = {w(\z)} (1+  {  F^m}(\z)^2). 
 \ee
 
 Let
 \be\lbl{p(t)=}
 p(t):= \int_{-\infty}^{+\infty}  w_2(v) h(q/L,v)  e(-tv) \,dv, \quad e(x):=e_1(x)=
 e^{2\pi i x}.
 \ee
 This is  the Fourier transform of function $w_2(\cdot)h(q/L,\cdot)$.
 Then, expressing $w_2 h$ via $p$ 
 by the inverse Fourier transform and writing $w(\z)=\tilde w(\z)w_2( {  F^m}(\z))$, we find that 
 $$
 w(\z)h(q/L, {  F^m}(\z))= \tilde w(\z)\int_{-\infty}^{+\infty} p(t)e(t {  F^m}(\z))\,dt.
 $$
 Inserting this representation into \eqref{tilde_I_q} we get
 \begin{equation*}
 \tilde I_q(\vc)=\int_{-\infty}^{+\infty} p(t){ e(-tm)}\,\left(\int_{\R^{d}} \tilde w(\z) e\big(tF(\z)-\vu\cdot \z\big) \, d\z \right)\,dt, \quad  \vu:=\vc\, L/q.
 \end{equation*}
 Note that 
 $$
 | \vu|   =|\vc| L/q> \alpha^{-1} >1
 $$ 
 since $q<\alpha |\vc| L $.   Now let us 
   denote
 $
 W_0(x) = c_0^{-d}\prod_{i=1}^{d}w_0(x_i) 
 $
 (see \eqref{eq:or_weight}). Then $W_0\in C_0^\infty (\R^d)$, $W_0\ge0$ and 
 \be\label{then}
  \supp W_0 =[-1, 1]^d \subset \{ x \in \R^d: |x| \le \sqrt{d} \},\quad \int_{\R^d} W_0(x)\,dx =1.
 \ee
 Let us set
 $
 \de  = | \vu|^{-1/2} < \sqrt\alpha
 $
 and write $\tilde w$ as 
 $$
 \tilde w(\z) = \de^{-d} \int_\R^d W_0\left( \frac{\z-\va}{\de}\right) \tilde w(\z)\,d\va. 
 $$
 Then  setting $\vb:=\displaystyle{\frac{\z-\va}{\de}}$  we get that 
 $$
 |\tilde I_q(\vc)|\leq \int_{\R^{d}}\int_{-\infty}^{+\infty}  |p(t)| 
 |I_{\va,t}|\, dt\,d\va,
 $$
 where in view of \eqref{then}, 
 \begin{equation*}
 I_{\va,t}:=\int_{\{ |\vb| \le \sqrt{d} \} }  W_0(\vb) \tilde w(\z) \, e(tF(\z) - \vu\cdot\z)\,d\vb, \qquad \z:=\va+\de\vb.
 \end{equation*} 
 Consider the exponent in the integral $ I_{\va,t}$:
 $$
 f(\vb)=f_{\va,t}(\vb):=tF(\va+ \de\vb) - \vu\cdot (\va+ \de\vb).
 $$
 At  the next step we will estimate  integral $I_{\va,t}$, 
 regarding  $(\va,t)$ as a parameter.
Consider another parameter $R$, satisfying
 $$\ds{1\leq  R\leq |\vu|^{1/3}};$$
its value will be chosen later.
Below we distinguish two cases:

 1. $(\va,t)$ belongs to the   "good" domain $S_R$, where 
 $$
 S_R = \big\{ ( \va,t): 
  |\nabla f(0)|=\de |t A\va - \vu| \ge 
  R\big\lan  t /  |\vu |\big\ran =R\lan \de^2t\ran \big\};
 $$ 
 
 2. $(\va,t)$ belongs to the  "bad" set ${S_R}^c = (\R^d \times \R) \setminus S_R$. 
 \smallskip

 \subsubsection{Integral over $S_R$.} 
 We consider first the integral over the good set~$S_R$: 
 \begin{lemma}\lbl{l:good_set}
 	For any {   $d\geq 1$}, $N\ge 0$ and $R\ge 2\|A\|\sqrt{d}$ we have
 	\be\lbl{int_to_bound}
 	\int_{S_R} |p(t)| \, |I_{\va, t} | \,  d\va\, dt
 	\lappr_{N,{ m}} \frac{L}{q}R^{-N}\|w\|_{N,d+5}\,.
 	\ee
 \end{lemma}
 
 {\it Proof.}
  Let $\vl:=\nabla f(0)/|\nabla
 f(0)|$  and  $\cL =\vl\cdot \nabla_\vb$.  Then for $(\va,t)\in S_R$ and $|\vb|\leq \sqrt{d}$ (see \eqref{then}), 
 \be\lbl{nabla_f}
 \begin{split}
   |\cL f(\vb)|& = \left|\cL f(0) +\delta^2 t \nabla
     f(0)\cdot A\vb/|\nabla f(0)|\right|
   \ge
   |\nabla f(0)| - \delta^2|t||A\vb|\\
&   \ge  R\lan\de^2 t \ran
 - \de^2|t|\|A\| \frac{R}{2\|A\|} \ge \tfrac12 R\lan\de^2 t \ran
 \geq R/2.
\end{split} \ee
 Since 
 $
 (2\pi i \cL f(\vb))^{-1} \cL e(f(\vb)) = e(f(\vb)), 
 $
 then integrating by parts  $N$ times  integral $ I_{\va,t}$   we get 
 $$
 |I_{\va,t}| \lappr_{N}
 \max_{|b_i|\le
 	1\, \forall i}\,\max_{0\leq k\leq N}\left|\cL^{N-k} \tilde
 w(\delta \vb +\va)\frac{\big(\cL^2f(\vb)\big)^k}{\big(\cL f(\vb)\big)^{N+k}}\right| \,,
 $$
 where we have used that $\cL^m f(\vb)=0$ for  $m\geq 3$. 
 Since $|\cL^2f(\vb)|\leq \de^2|t||\vl\cdot A\vl|\leq\de^2|t|\|A\|$, then  in view of \eqref{nabla_f} 
 $$
 \left|\frac{\cL^2 f(\vb)}{\cL f(\vb)}\right|  \leq\frac{\de^2|t|\|A\|}{\tfrac12 R\lan\de^2 t \ran }=\frac{2\|A\|}{R}\leq \frac{1}{\sqrt{d}}.
 $$
 So using that $\ds{\Big|\frac{1}{\cL f(\vb)}\Big|\leq\frac2R}$ by \eqref{nabla_f}, we find
 $$
 |I_{\va,t}| \lappr_{N} R^{-N}
 \max_{|b_i|\leq 1 \,\forall i}\,\max_{0\leq k\leq N}\left|\cL^{k} \tilde
 w(\delta \vb +\va)\right|. 
 $$
 Thus, denoting by  ${\bf 1}_{S_R}$  the indicator function of the set $S_R$, we have 
 \be\non
 \begin{split}
 	\int_{\R^{d}}|I_{\va,t}|{\bf 1}_{S_R}\,d\va
 	&\lappr_{N} R^{-N}\!\int_{\R^{d}}  
 	\Big(\lan\va\ran^{d+1} \!\!
 	\max_{|b_i|\le 1\,\forall i}\, \max_{0\leq k\leq N}\big|\cL^k\tilde
 	w(\delta \vb +\va)\big| \Big)\, \frac{d\va}{\lan\va\ran^{d+1}}
 	\\
 	& \lappr_{N} R^{-N}\|\tilde w\|_{N,d+1}
        \lappr_{N,{ m}} R^{-N}\|w\|_{N,d+5}  \,,
 \end{split}
 \ee
 for every $t$.  Then 
 \be
 \mbox{l.h.s. of \eqref{int_to_bound}}\lappr_{N,m} R^{-N}\|w\|_{N,d+5} \int_{-\infty}^{+\infty} |p(t)| \, dt.
 \ee

 To prove  \eqref{int_to_bound} it   remains to show that
 \begin{equation}\label{intp(t)}
 \int_{-\infty}^\infty|p(t)|dt\lappr L/q\,.
 \end{equation}
 In virtue of Lemma~\ref{l:5} with $N=2$,
 $$
 \Big|\frac{\p^k}{\p v^k} h(x,v)\Big| \lappr_k x^{-k-1}\min\{1,x^2/v^2\}\,,\quad k\geq 1,
 $$
 and by Corollary~\ref{c:h}, $|h(x,v)|\lappr x^{-1}$.
 Then an integration by parts in \eqref{p(t)=} shows that,
 for any $M\geq 0$, 
 \[\begin{split}
 |p(t)| & \lappr_M |t^{-M}|
 \Big( 
 \int_{-\infty}^\infty
 |w_2^{(M)}(v)|x^{-1}\,dv \\
 &+  \max_{1\leq k \leq M}\int_{-\infty}^\infty
 |w_2^{(M-k)}(v)| \, x^{-k-1}\min\big\{1,\frac{x^2}{v^2}\big\}\,dv\Big),
 \end{split}
 \] 
 where $x:=q/L$. Writing  the latter integral as a sum $\int_{|v|\leq x} + \int_{|v|>x}$ we see that
 $$
 \int_{|v|\leq x}= x^{-k-1}\int_{|v|\leq x} |w_2^{(M-k)}(v)|\,dv \lappr_{{  M}} x^{-k}
 $$
 and 
 $$
 \int_{|v|>x} = x^{-k+1}\int_{|v|>x} \frac{ |w_2^{(M-k)}(v)|}{v^2}\,dv \lappr_{  M} x^{-k}.
 $$
 Then, for any $M\ge 0$
 \begin{equation}\lbl{p(t)<}
 |p(t)| \lappr_M \left(\frac{q}{L}|t|\right)^{-M}\qmb{if}\qu \frac qL <1 \qmb{and}\qu 
 |p(t)| \lappr_M \left(\frac{q}{L}\right)^{-1} |t|^{-M} \qmb{if}\qu \frac qL \ge 1.
 \end{equation}
 Choosing $M=2$ when $|t|>\lan L/q \ran$ and $M=0$ when $|t|\leq \lan L/q \ran$ we get \eqref{intp(t)}.
 \qed

 \subsubsection{Integral over  ${S_R}^c$.} 
 
 Now we study the integral over the bad set ${S_R}^c$.
 \begin{lemma}\label{l:bad_set}
 	For any {   $d\geq 1$}, $1\le R\le |\vu|^{1/3}$ and {  $0<\beta<1$ }we have
 	$$
 	\int_{S_R^c} |p(t)| |I_{\va, t}| 
 		\, d\va \,dt
 	\lappr_{{m}}  R^{d}|\vu|^{-d/2+1+\beta}\|w\|_{0,K(d,\beta)}\,,
 	$$
 	where $K(d,\beta)=  d+\lceil d^2/2\beta\rceil+4$.
 \end{lemma}
 
 {\it Proof.}
 On ${S_R}^c$ we use for $I_{\va, t}$ the easy upper bound
 \begin{equation}\lbl{I(a,t)<}
 |I_{\va, t}|\lappr\max_{|b_i|\le 1\,\forall i}|\tilde w(\de \vb +\va)|\ \leq \|\tilde w\|_{0,0}.  
 \end{equation}
 The fact that $(\va,t)\in {S_R}^c$ implies that the integration in 
  $d \va$ for a  fixed $t$ is restricted to the region, where
$
 \left|A\va - t^{-1} {\vu} \right| \le  ({R}/{\de |t|}) \lan  t / |\vu| \ran \,,
 $
 or
  \begin{equation}\label{bound_int_a}
 \left|\va -\frac{A^{-1}\vu}{t}\right|\le  \|A^{-1}  \| \, \frac{R}{\de |t|} \, \lan  t / |\vu| \ran \,.
 \end{equation}
 We first consider  the case $|t|\ge |\vu|^{1-\beta/d}$. Since $|\vu| >1$, then considering separately the cases $|t| \le |\vu|$ and 
  $|t| \ge |\vu|$ we see that 
 \be\label{grr}
 \frac{R}{\de
 	|t|}\lan  t / |\vu| \ran \le 
	 R|\vu|^{-1/2+\beta/d}\,.
\ee
 In view of  \eqref{I(a,t)<} -\eqref{grr},
 $$
 \left|\int_{\R^{d}} |I_{\va,t}|{\bf 1}_{{S_R}^c}(\va,t)d\va\right|\lappr
 {R^{d}}|\vu|^{-d/2+\beta} \|\tilde
 w\|_{0,0} \,.
 $$
 Since $|F^m(\z)|\lappr_m \lan \z \ran^2$, by definition \eqref{tilde_w_def} of the function $\tilde w$ we have $\|\tilde w\|_{0,0}\lappr_m \|w\|_{0,4}$. Then the r.h.s. above is $\lappr_m {R^{d}}|\vu|^{-d/2+\beta} \|  w\|_{0,4}$. Taking into account that, by  \eqref{intp(t)}, 
 $
 \ds{\int_{|t|\ge |\vu|^{1-\beta/d}} |p(t)|\,dt
  \lappr \frac{L}{q}\le|\vu|\,,}
 $
 we get
 \begin{equation}\label{t>u}
\int_{|t|\ge |\vu|^{1-\beta/d}}\left(\int_{\R^{d}} |p(t)| |I_{\va, t}|{\bf 1}_{{S_R}^c}(\va,t) 
  \,d\va \right) dt
 \lappr_{{m}}
 R^{d}|\vu|^{-d/2+1+\beta} \|  w\|_{0,4}\,.
 \end{equation}
 
 Now let  $|t| \le |\vu|^{1-\beta/d}$. Then  the r.h.s. of \eqref{bound_int_a}  is bounded by the quantity 
 $
  \|A^{-1}  \|  {R}/{(\de |t|)}, 
 $ 
  so that $|\va|\gtrsim |A^{-1}\vu|/|t|- \|A^{-1}  \|  {R}/{(\de |t|)}.$ Since $|A^{-1}\vu|\geq C_A|\vu|$ and $R\le  | \vu |^{1/3}$, then 
 $$
 |\va|\gtrsim_A
  \frac{|\vu|- R {  C'_A}\sqrt{|\vu|}}{|t|}\ge (1-{  C'_A}|\vu|^{-1/6})\frac{|\vu|}{|t|} \ge
 \frac12\frac{|\vu|}{|t|}\ge \frac 12 |\vu|^{\beta/d}\, 
 $$
 with $C_A'=C_A^{-1}\|A^{-1}\|$, since $|\vu|^{-1}\leq \alpha$, 
 if $\alpha$  is so small that $1-C'_A\alpha^{1/6}\geq 1/2$.
 Then $1\lappr |\va|/|\vu|^{\beta/d}$ on ${S_R}^c$, so
   that ${\bf 1}_{{S_R}^c}(\va,t) \lappr   |\vu|^{-d/2+\beta/d}  |\va|^{d^2/2\beta-1}$, and  we deduce from
 \eqref{I(a,t)<} that for such values of $t$
 \begin{equation*}
 \begin{split}
 \left|\int_{\R^d} |I_{\va,t}| {\bf 1}_{{S_R}^c}(\va,t) d\va\right|&\lappr |\vu|^{-d/2+\beta/d}
 \int_{\R^{d}} |\va|^{d^2/2\beta-1} \max_{|b_i|\le 1\,\forall i}|\tilde w(\de \vb +\va)|\, d\va
 \\
 &\lappr_{{m} } |\vu|^{-d/2+\beta/d} \|w\|_{0,K(d,\beta)}\,, 
 \end{split}
 \end{equation*}
 where $K(d,\beta)=d+\lceil d^2/2\beta\rceil+4$.
 On the other hand, by \eqref{p(t)<} with $M=0$, 
 $
 \int_{|t|\le |\vu|^{1-\beta/d}}
  |p(t)|dt \lappr |\vu|^{1-\beta/d}\,,
 $
 from which we obtain
 \begin{equation}\label{t<u}
\int_{|t|\le |\vu|^{1-\beta/d}}\left(\int_{\R^{d}}|p(t)| |I_{\va, t}| {\bf 1}_{{S_R}^c}(\va,t)
  \,d\va\right)dt
  \lappr_{{m}}
 |\vu|^{-d/2+1} \| w\|_{0,K(d,\beta)}\,. 
 \end{equation}
 
 Putting together \eqref{t>u} and \eqref{t<u}  we get the assertion. 
 \qed

\subsubsection{End of the proof }  
 In order to complete the proof of Lemma~\ref{l:22} we combine 
 Lemmas~\ref{l:good_set} and \ref{l:bad_set} to get that 
 $$
 |\tilde I_q(\vc)|\lappr_{N{ ,m}} \left(\frac Lq R^{-N} +R^{d} |\vu|^{-d/2+1+\beta}\right)\left(
 \|w\|_{N,d+5}+\|w\|_{0,K(d,\beta)}\right)\,.
 $$
 We fix here {  $\gamma_1\in (0,1/2)$}, $\beta = {\gamma_1}/2$, $R= |\vu|^{\frac{{\gamma_1}}{2d}}\le|\vu|^{\frac{1}{3}}$ 
 and pick $N= \lceil\tfrac{d^2}{{\gamma_1}}\rceil-2d>0$  (notice that $R \ge \alpha^{-\gamma_1/2d}
 \ge 2\|A\|\sqrt{d}$ if
 $\alpha$ is small enough, so that assumption of Lemma~\ref{l:good_set} is satisfied). 
 Then $K(d,\beta)=N+3d+4$, $R^{-N}\leq|\vu|^{-d/2+\gamma_1}\leq  |\vc| \,(L|\vc|/q)^{-d/2+\gamma_1}  $  since  $|\vc|\ge 1$. Moreover, $R^d |\vu|^{-d/2+1+\beta}=|\vu|^{-d/2+1+\gamma_1} =   (L  |\vc|/q)^{-d/2+1+\gamma_1}$. 
 This concludes the proof.
 \qed

\section{Integrals over quadrics}\label{sec:7}
Our goal in this section is to study integrals $\II(t;w)$ over the quadrics $\Sigma_t$. We start with a case
of  quadratic forms $F$, written in a convenient normal form (Theorem~\ref{t_ap1}),  and 
 show later in  Section~\ref{s_7.4}
 (Theorem~\ref{c_integrals})   how to reduce general integrals $\II(t;w)$ to those, corresponding 
to the quadratic  forms like that. In this section we assume that
$$
d\ge3
$$
and  not use the bold font to denote vectors since most 
 of variables we use are vectors.

\subsection{Quadratic forms in normal form}
On
 $\R^{d}=\R^n_u\times \R_x^{d_1}\times \R_y^{d_1}= \{z=(u,x,y)\}$,
where $d\ge3$,  $n\ge 0$ and ${d_1}\ge 1$, consider the
quadratic form
\begin{equation}\label{p-1}
    F(z)=\tfrac12|u|^2+x\cdot y =\tfrac12 Az\cdot z\,,\quad
    A(u,x,y)=(u,y,x)\,.
\end{equation}
Note that $A$ is an orthogonal operator, $|Az|=|z|$. As in
Section~\ref{s_1.1} we define the quadrics $\Sigma_t=\{z:F(z)=t\}$,
$t\in \R$. Note that for $t\neq 0$ $\Sigma_t$ is a smooth
hypersurface, while $\Sigma_0$ is a cone with a singularity at the
origin. We denote the volume element on $\Sigma_t$ (on
$\Sigma_0\backslash\{0\}$ if $t=0$), induced from $\R^{d}$, as
$dz\!\mid_{\Sigma_t}$ and set
\begin{equation}\label{p0}
  \mu^{\Sigma_t}(dz)=|Az|^{-1}dz\!\mid_{\Sigma_t}\,
\end{equation}
(see below concerning this measure when $t=0$). 

For a  $k_*\in \N\cup\{0\}$ and a function $f$ on $\R^{d}$ satisfying
\begin{equation}\label{p00}
f\in\cC^{k_*,M}(\R^d)\,, \quad  M>{d}\,,
\end{equation}
we will study the integrals
\be\lbl{integrals}
\II(t)=\II(t;f)= \int_{\Sigma_t} f(z)\mu^{\Sigma_t}(dz)\,.
\ee

Our first goal is to demonstrate the following result: 

\begin{theorem}\lbl{t_ap1}
  For the quadratic form  $F(z)$ as in \eqref{p-1} and  a function $f\in\cC^{k_*,M}(\R^d)$, $M>{d}$, consider  integral
  $\II(t;f)$, defined in \eqref{integrals}. Then  the function $\II(t)$, defined by \eqref{integrals}, 
  is $C^k$--smooth  if $k<  {d}/2-1$, $k\le k_*$, 
  and is  $C^k$--smooth outside zero if  $k\le  \min(  {d}/2-1,k_*)$.
  For  $ 0<|t|\le 1$ we have 
      \be\label{log}
      \begin{split}
      &\left| \p^k\II(t)\right|{ \lappr_{k,{ M}}}\|f\|_{k,{M}}\, \quad \text{if}\;\;  k< {d}/2-1,\\
     & \left| \p^k\II(t)\right|{ \lappr_{k,{M}}}\|f\|_{k,{M}}(1-\ln |t|) \quad \text{if}\;\;    k\le {d}/2-1.
      \end{split}
      \ee
      While for $|t|\ge 1$, denoting $ \kappa= \frac{M+2-{d}}2$, we have 
      \be\lbl{q66}
      \begin{split}
      \left| \p^k\II(t)\right| &{ \lappr_{k,M}}\|f\|_{k,M}\lan t\ran^{-\kappa}\, \quad 
     \; \text{if}\; 1\le  k\le  {d}/2-1,\;   k\le k_*,
     \\
       \left| \II(t)\right| &{ \lappr_{M,\kappa'}}\|f\|_{0,M}\lan t\ran^{-\kappa'}\, \quad \forall\, \kappa'<\kappa. 
      \end{split} 
      \ee
\end{theorem}
An example, see \cite[Example A.3]{arXiv}, shows that in general  the log-factor cannot be removed from the r.h.s. in \eqref{log}.

The theorem is proved below in number of steps. 
In the proof for a given vector $x\in \R^{d_1}$ we consider its orthogonal complement in $\R^{d_1}$ -- the hyperspace $x^\perp$. We denote its elements $\bar x$,  and provide $x^\perp$  with the Lebesgue measure  $d\bar x$. 
If $d_1=1$, then    $x^\perp$ degenerates to the space $\R^0= \{0\}$, and 
$d\bar x$ -- to the $\delta$-measure at $0$. Practically it means that when $d_1=1$, the spaces $x^\perp$ and $y^\perp$ (and integrals over them) 
disappear from our construction. It makes the case $d_1=1$ easier, but notationally different from $d_1\ge2$. For example, in formula \eqref{sigm^x} 
with $d_1=1$ the affine space $\sigma_t^x(u', x')$ becomes the point $(u', x', (t- \frac12 |u'|^2) |x'|^{-2} x')$, the measure $d\mu^{\Sigma_t}\!\mid_{\Sigma_t^x}$
in \eqref{p05} becomes $du\, |x|^{-1} dx$, etc. Accordingly, below we write the proof only for  $d_1\ge2$, leaving the case $d_1=1$ as an easy 
exercise for the  reader.

\subsection{Disintegration of the two measures}
Our goal in this subsection is to find a convenient disintegration of
the measures $dz\!\mid_{\Sigma_t}$ and $\mu^{\Sigma_t}$, following the
proof of Theorem 3.6 in \cite{DK1}.

Recall that we write
elements $z\in\R^d$ as $z=(u,x,y)$,  where $u\in\R^d$ and $x,y\in \R^{d_1}$. 
Let us denote $\Sigma_t^x= \{(u,x,y)\in \Sigma_t: x\neq 0\}$ (if
$t<0$, then $\Sigma_t^x=\Sigma_t$). Then for any $t$  \ $\Sigma_t^x$ is a smooth hypersurface in $\R^d$, and 
the mapping
\begin{equation}\label{p2}
\Pi_t^x:\Sigma_t^x\to\R^n\times \R^{d_1}\backslash\{0\}\,,\quad
(u,x,y)\mapsto  (u,x)\,,
\end{equation}
is a smooth affine euclidean vector bundle. Its fibers are
\be\label{sigm^x}
\sigma^x_t(u',x'):=\big(\Pi_t^x\big)^{-1}(u',x') = \Big(  u', x',
      {x'}^\perp + \frac{t-\tfrac12|u'|^2}{|x'|^2} x'\Big)\,,
\ee
where ${x'}^\perp$ is the orthogonal complement to $x'$ in
$\R^{d_1}$. For any $x'\neq 0$ denote
$$
U_{x'}= \big\{x:|x-x'|\le \frac12 |x'|\big\}\,,\quad U = \R^n\times
U_{x'} \times \R^{d_1}\,. 
$$

Now we construct a trivialisation of the bundle $\Pi_t^x$ over $U$. To do this we
fix in $\R^{d_1}$ any orthonormal frame $(e_1,\ldots,e_{d_1})$ such that
the ray $\R_+ e_1$ intersects $U_{x'}$. Then
$$
x_1>0 \quad \forall x= (x_1,\ldots,x_{d_1})=:(x_1,\bar x) \in
U_{x'}\,.
$$
We wish to  construct an affine in the third argument diffeomorphism 
$$
\Phi_t: \R^n\times U_{x'}\times \R^{{d_1}-1}\to U\cap \Sigma_t
$$
of the form
\begin{equation}\label{p3}
\Phi_t(u,x, \bar \eta) = (u,x,\Phi_t^{u,x}(\bar \eta))\,, \quad
\Phi_t^{u,x}(\bar \eta) = (\varphi_t(u,x, \bar \eta),\bar \eta) \in \R^{d_1} 
\,, \quad \bar\eta\in \R^{d_1-1}. 
\end{equation}
We easily see that $\Phi_t(u,x,\bar \eta)\in \Sigma_t$ if and only if
\begin{equation}\label{p4}
  \varphi_t(u,x,\bar \eta) = \frac{t-\tfrac12|u|^2-\bar x\cdot\bar
    \eta}{x_1}\,.
\end{equation}
The  mapping $\bar\eta\to\Phi_t^{u,x}(\bar\eta)$ with this function $ \varphi_t$ 
 is affine,  and the range of $\Phi_t$  equals $U\cap \Sigma_t$.

In the coordinates $(u,x,\eta_1,\bar\eta)\in \R^n\times U_{x'}\times \R\times \R^{{d_1}-1}$ on the domain 
$U \subset \R^{d}$ 
the hypersurface $\Sigma_t^x$ is embedded in
$\R^{d}$ as a graph of the function $(u,x,\bar \eta)\mapsto \eta_1 =
\varphi_t$. Accordingly, in the coordinates $(u,x,\bar\eta)$ on $U\cap
\Sigma_t$ the volume element on $\Sigma_t$ reads as $\bar
\rho_t(u,x,\bar \eta)du\,dx\,d\bar \eta$, where
\begin{equation*}
  \begin{split}
\bar \rho_t  = \left(1+|\nabla \varphi_t|^2\right)^{1/2}
= \Big(1+\frac{|u|^2+|\bar \eta|^2+|\bar x|^2 + x_1^{-2}(t-\tfrac12
  |u|^2 -\bar x\cdot \bar \eta)^2}{x_1^2}\Big)^{1/2}\!.
\end{split}
\end{equation*}
Passing from the variable $\bar \eta\in \R^{{d_1}-1}$ to
$y=\Phi_t^{u,x}(\bar \eta) \in \sigma_t^x(u,x)$ we replace $d\bar
\eta$ by $|\det\Phi_t^{u,x}(\bar \eta)|d_{\sigma_t^x(u,x)}y$. Here
$d_{\sigma_t^x(u,x)}y$ is the Lebesgue measure on the $(d_1-1)$-dimensional affine euclidean 
space  $\sigma_t^x(u,x)$  while  by $\det\Phi_t^{u,x}$ we denote the determinant of the linear mapping  $\Phi_t^{u,x}$, viewed as a linear isomorphism of the euclidean space $\R^{d_1-1} = \{\bar\eta\}$ and the tangent space to  $\sigma_t^x(u,x)$,  identified 
with the euclidean space $x^\perp\subset \R^{d_1}$. 
Accordingly we write the volume element on $\Sigma_t \cap U$ as 
$
\rho_t(u,x,y) du\,dx\,d_{\sigma_t^x(u,x)}y 
$
with
$$
 \rho_t(u,x,y) = \bar
\rho_t(u,x,\bar \eta)  |\det\Phi_t^{u,x}(\bar \eta)|\,, \quad (u,x,y) \in \Sigma_t,\; \text{where}\;
\Phi_t^{u,x}(\bar \eta)=y. 
$$

Now we will calculate the density $\rho_t$. Let us take any point
$z_*=(u_*,x_*, y_*)\in U\cap \Sigma_t $ and choose a frame $(e_1,\ldots,e_{d_1})$ such
that $e_1=x_*/|x_*|$. Then
$$
x_*= (|x_*|,0)\,, \;\; y_*= \big(y_{*1}, \bar y_*\big)\,, 
\qquad y_{*1}= \Big(\frac{t-\tfrac12|u_*|^2}{|x_*|} \Big), \;\;
 \bar y_*\in \R^{{d_1}-1}\,.
$$
So (see~\eqref{p3}--\eqref{p4}) the mapping $\Phi_t$ is such that
$
\Phi_t^{u_*,x_*}(\bar \eta) = \big(y_{*1}, 
\bar \eta\big)=\tilde y\in\sigma_t^x(u_*,x_*)
$
(i.e. $\varphi_t(z_*) = y_{*1}$). 
In these coordinates 
$ \rho_t(u_*,x_*,y_{*1}, \bar y_*)=\bar \rho_t(u_*,x_*, \bar y_*) $,
which equals 
\begin{equation*}
  \begin{split}
\left(1+|x_*|^{-2}\left(|u_*|^2+|\bar y_*|^2
+|y_{*1}|^2\right)\right)^{1/2}
&=\frac{\left(|x_*|^2 +|u_*|^2+|\bar y_*|^2
+|y_{*1}|^2\right)^{1/2}}{|x_*|}\,.
  \end{split}
\end{equation*}
That is, $\rho_t(z_*) = \tfrac{|z_*|}{|x_*|}$. Since $z_*$ is any point in $U\cap \Sigma_t$, then 
 we have proved
\begin{proposition}\label{p_1}
  The volume element $dz\!\mid_{\Sigma_t^x}$ with respect to the
  projection $\Pi_t^x$ disintegrates as follows:
  \begin{equation}\label{p5}
    dz\!\mid_{\Sigma_t^x} = du\,|x|^{-1}dx\,|z|d_{\sigma_t^x(u,x)}y\,.
  \end{equation}
  That is, for any function $f\in C_0^0(\Sigma_t^x)$,
  $$
\int f(z) dz\!\mid_{\Sigma_t^x}= \int_{\R^n}\int_{\R^{d_1}}
|x|^{-1}\Big(\int_{\sigma_t^x(u,x)} |z|f(z) \, d_{\sigma_t^x(u,x)}y\, \Big)dx\, du\,.
  $$
\end{proposition}

Similarly, if we set $\Sigma_t^y= \{(u,x,y)\in \Sigma_t: y\neq 0\}$
and consider the projection
$$
\Pi_t^y:\Sigma_t^y\to\R^n\times \R^{d_1}\backslash\{0\}\,,\quad
(u,x,y)\mapsto  (u,y)\,,
$$
then
\begin{equation}\label{p55}
  dz\!\mid_{\Sigma_t^y} = du\,|y|^{-1}dy\,|z|d_{\sigma_t^y(u,y)}x\,.
\end{equation}

Let us denote $\Sigma_t^0= \{(u,x,y)\in \Sigma_t: x=y= 0\}$. Then
$\Sigma_t\backslash \Sigma_t^0$ is a smooth manifold and
$dz\!\mid_{\Sigma_t}$ defines on it a smooth measure.

By \eqref{p5} and \eqref{p55}, for any $t$   the volume of the set
$\{z\in\Sigma_t\backslash \Sigma_t^0 : 0 < |x|^2+|y|^2\le\eps\}$ goes
to zero with $\eps$. So  assigning to $\Sigma_t^0$ zero measure we
extend $dz\!\mid_{\Sigma_t}$ to a Borel measure on $\Sigma_t$ such that
each set $\{z\in \Sigma_t: |z|\le R\}$ has a finite measure and
\begin{equation}\label{p555}
  \left(dz\!\mid_{\Sigma_t}\right)\left(\left(\Sigma_t^x\cup \Sigma_t^y\right)^c\right) = 0\,.
\end{equation}

By \eqref{p5} and \eqref{p55}  function
$|z|^{-1}$ is locally integrable $\Sigma_t$ with respect to the measure 
$dz\!\mid_{\Sigma_t}$. So $\mu^{\Sigma_t}$ (see \eqref{p0}) is a well
defined Borel measure on $\Sigma_t$. Since $|Az|=|z|$, then, in view of
\eqref{p5} and \eqref{p55},
\be\label{p05}
d\mu^{\Sigma_t}\!\mid_{\Sigma_t^x}=
du\,|x|^{-1}dx\,d_{\sigma_t^x(u,x)}y\,,\quad 
d\mu^{\Sigma_t}\!\mid_{\Sigma_t^y}=
du\,|y|^{-1}dy\,d_{\sigma_t^y(u,y)}x\,.
\ee

The measure $\mu^{\Sigma_t}$ defines on $\R^{d}$  a Borel measure,  supported by $\Sigma_t$. It will also be denoted  $\mu^{\Sigma_t}$.

\subsection{Analysis of the integral $\II(t;f)$}
Note that for any $t$ the mapping
$$
L_t : \Sigma_0^x\to \Sigma_t^x\,,\quad (u,x,y)\mapsto (u,x,y+t|x|^{-2}x)
$$
defines an affine isomorphism of the bundles
$\Pi_0\!\mid_{\Sigma_0^x}$ and $\Pi_t\!\mid_{\Sigma_t^x}$. Since $L_t$
preserves the Lebesgue measure on the fibers, then in view of
\eqref{p5} it  sends the measure $\mu^{\Sigma_0}$ to
$\mu^{\Sigma_t}$. Using \eqref{p05} we get that for any $t$ 
 the integral $ \II(t)$, defined in \eqref{integrals}, may be written as 
\begin{equation}\label{p6}
  \begin{split}
    \II(t;f)&\int_{\Sigma_0}f(L_t(z))\mu^{\Sigma_0}(dz)\\
    &=\int_{\R^n\times \R^{d_1}} |x|^{-1}\Big(\int_{\sigma(u,x)} f(u,x,
    y+t|x|^{-2}x )d_{\sigma^x(u,x)}y \Big)du\,dx\,.
  \end{split}
\end{equation}
Here
$
\sigma(u,x):=\sigma_0^x(u,x)=x^\perp-\frac12|u|^2|x|^{-2}x\,.
$

We recall that $f(u,x,y)$ satisfies \eqref{p00}. Taking any smooth function
$\varphi(t) \ge 0$ on $\R$ which vanishes for $|t|\ge 2$ and equals
one for $|t|\le 1$ we write
$
f=f_{00}+f_1$, where  $ f_{00} = \varphi(|(x,y)|^2)f$
and 
$ f_{1} =(1- \varphi(|(x,y)|^2))f$.
Denoting $B_r(\R^m)= \{\xi\in \R^m: |\xi|\le r\}$ and  $B^r(\R^m)=
\{\xi\in \R^m: |\xi|\ge r\}$ we see that
\begin{equation}\label{p7}
  \supp f_{00}\subset \R^n\times B_{\sqrt2}(\R^{2{d_1}})\,, \quad  \supp
  f_1\subset \R^n\times B^1(\R^{2{d_1}})\,.
\end{equation}
Setting next $f_{11}(z) = f_1(z)(1-\varphi(4|x|^2))$, $f_{10}(z)=
f_1 (z)\varphi(4|x|^2)$ we write 
$$
f= f_{00}+f_{11} +f_{10}\,.
$$
Since $(x,y)\in B^1(\R^{2{d_1}})$ implies that $|x|\ge 1/\sqrt{2}$ or
$|y|\ge 1/\sqrt{2}$, then in view of \eqref{p7},
\begin{equation}\lbl{p8}
  \begin{split}
  \supp f_{11}&\subset \R^n\times B^{1/2}(\R^{d_1}_x)\times \R^{d_1}_y\,,\\   \supp
  f_{10}&\subset \R^n\times \R^{d_1}_x\times B^{1/\sqrt2}(\R^{d_1}_y)\,.
  \end{split}
\end{equation}
Obviously, for  $i,j=0,1$ we have  
$\|f_{ij}\|_{k,m}\le C_{k,m} \|f\|_{k,m}$, for all $k\leq k_*$, $m\leq M$. 

Setting $\II_{ij}(t)= \II(t;f_{ij})$ we have: 
$$
\II(t;f)= \II_{00}(t)+\II_{10}(t) +\II_{11}(t)\,.
$$

\subsubsection{Integral $\II_{00}(t)$.} \label{s_7.3.1} 
By \eqref{p6} $\II_{00}(t)$ is a 
continuous function, and for $1\le k\le k_*$,
\begin{equation}\label{p9}
  \begin{split}
 & \partial^k \II_{00}(t) = \int_{\R^n}\Big(\int_{B_{\sqrt2}(\R^{d_1})}
  |x|^{-1} dx \Big) du \\
  &\qquad\int_{y\in \sigma(u,x)}
  \big({d^k}/{dt^k}\big)
  f_{00}(u,x,y+t|x|^{-2} x) \, d_{\sigma(u,x)}y\\
  &=\int_{\R^n}\!\int_{B_{\sqrt2}(\R^{d_1})}\!\!
  |x|^{-1}  \Big( \int_{y\in \sigma(u,x)}\!\!
  d_y^k
  f_{00}(u,x,y+t|x|^{-2} x) \left[|x|^{-2}x\right] \, d_{\sigma(u,x)}y \Big)dx du ,
  \end{split}
\end{equation}
 where by $d_y^k
	f_{00} \left[|x|^{-2}x\right]$ we denote the action of the differential $d_y^k f_{00}$ on the set of $k$ vectors, each of which equals to $|x|^{-2}x$.
Setting $\tau=t-\frac12|u|^2$,  for $y\in\sigma(u,x)$ we have
\be\lbl{y_in_sigma}
y+t|x|^{-2} x = \bar y+\tau|x|^{-2}x, \qquad \mbox{for
    some }\bar y\in x^\perp.
\ee
Then we write the integral over $y$ in \eqref{p9} as 
\be\lbl{p99}
\int_{x^\perp} d_y^k f_{00}(u,x,\bar y+\tau|x|^{-2}x) \left[|x|^{-2}x\right] 
\,d\bar y.
\ee
Since $|\bar y+\tau x|x|^{-2}|^2=|\bar y|^2 + \tau^2|x|^{-2}$, then on
the support of the integrand 
\be\lbl{p10}
|x|\leq \sqrt2, \qquad |\bar y|^2+ \tau^2 |x|^{-2}  \leq 2.
\ee
In particular, 
\be\lbl{p11}
|\tau|=\big|t-\frac12|u|^2\big|\leq \sqrt 2 |x|\le 2 \qmb{in \eqref{p99}}.
\ee
By \eqref{p7} the diameter of the domain of integration in \eqref{p99}
is bounded by $\sqrt2$. So, for any $m\geq 0$ integral \eqref{p99} is
bounded by $C_{k,m}  |x|^{-k}\lan u\ran^{-m}\|f\|_{k,m}$.
 Denoting $R=|u|$, $r=|x|$ we get that
\be\lbl{p12}
|\p^k \II_{00}(t)| \lappr_{k,M} \|f\|_{k,M} 
\int_{0}^{\sqrt2} r^{{d_1}-k-2} \Big(\int_0^\infty R^{n-1}\lan
R\ran^{-M}\chi_{|\tau|\le \sqrt 2 r} \,  dR\Big) dr  \,.
\ee
If $n=0$, then the integral over $R$ should be removed from the r.h.s.
 Below we estimate the integral $\p^k \II_{00}(t)$ separately for the cases $n=0$ and $n\ge 1$.

a) If $n=0$, then $\tau=t$, we get from \eqref{p11} that $|x|\geq
t/\sqrt2 $ and see from \eqref{p7} that, for $t\ne 0$, $\II_{00}(t)$
is $C^{k_*}$-smooth (since $f\in C^{k_*}$).  Then from \eqref{p12}
  we obtain
\be\lbl{p122}
|\p^k \II_{00}(t)| \lappr_{k,M} \|f\|_{k,M} 
\int_{|t|/\sqrt2}^{\sqrt2}r^{{d_1}-k-2} \chi_{|t|\leq 2}\, dr   \,.
\ee
	From here it follows that
	\be\lbl{p13}
	\begin{split}
		|\p^k \II_{00}(t)| \lappr_{k}  \|f\|_{k,M} \qmb{if} \qu  k\leq  \min({d_1}-2, k_*),\\
		|\p^k \II_{00}(t)| \lappr_{k}  \|f\|_{k,M}\big(1+\big|\ln |t|\big|\big) \qmb{if} \qu
		k\leq  \min({d_1}-1, k_*)\ ,
	\end{split}
	\ee
	while $\II_{00}(t)=0$ for $|t|\geq 2$.

b) If $n\geq 1$, then to estimate $\p^k \II_{00}(t)$ we split the integral for $ \II_{00}(t)$ in a sum of two. Namely, for a 
  fixed $t\neq 0$ we write  $f_{00} $  as 
    $f_{00} = f_{00<}+f_{00>}$, with   $f_{00<}=f_{00}\varphi(8|x|^2/t^2)$, where  $\varphi$ is the function, used to define
    the functions $f_{ij}$, $0\le i,j\le1$. 
     Then
  \be\label{supp}
  \supp f_{00<} \subset \{ 2|x| \le |t|\}, \quad  \supp f_{00>} \subset \{ 2\sqrt2 |x| \ge |t|\}.
  \ee
  With an obvious notation  we
  have $\II_{00}(t)=\II_{00<}(t)+\II_{00>}(t)$, where
  \begin{equation*}
    \begin{split}
  \II_{00<}(t) = & \int_{\R^n}\int_{B_{\sqrt2}(\R^{d_1})\cap B_{|t|/2}(\R^{d_1})}
  |x|^{-1}  \\
  &\qquad\Big(\int_{\substack{y\in \sigma(u,x)\\|x|^2+|y+t|x|^{-2}x|^2\le2}}
  f_{00<}(u,x,y+t|x|^{-2} x) \, d_{\sigma(u,x)}y \Big) dxdu \,, \\
  \II_{00>}(t) = & \int_{\R^n}\int_{B_{\sqrt2}(\R^{d_1})\cap B^{|t|/2\sqrt2}(\R^{d_1})}
  |x|^{-1}  \\
  &\qquad\Big(\int_{\substack{y\in \sigma(u,x)\\|x|^2+|y+t|x|^{-2}x|^2\le2}}
  f_{00>}(u,x,y+t|x|^{-2} x) \, d_{\sigma(u,x)}y \Big)dx du \,.
    \end{split}
  \end{equation*}

  Consider first  function   $\II_{00<}(t)$.   We observe that,  by \eqref{y_in_sigma},  for $y\in \sigma(u,x)$ and $|x|\le |t|/2$ (cf. \eqref{supp})
  $$
|y+t|x|^{-2}x|\ge { |\tau||x|^{-1}} = \left|t-\frac12|u|^2\right||x|^{-1}\ge-t|x|^{-1}>
\sqrt 2  \,,\quad \mbox{for } t<0  \,, 
  $$
so that  $\cI_{00<}(t)= 0$ for
$t<0$. For $t>0$,  performing the change of 
  variables $\sqrt{t} 
  u'=u$, $tx'=x$, we get
  \begin{equation*}
\begin{split}
 \II_{00<}(t) = & t^{{d}/2-1}\int_{\R^n}\int_{B_{\sqrt2/t}(\R^{d_1})\cap B_{1/2}(\R^{d_1})}
  |x'|^{-1} \varphi(8|x'|^2)  \\
  &\;\Big(\int_{\substack{y\in \sigma(u',x')\\|x'|^2t^2+|y+|x'|^{-2}x'|^2\le2}}
  f_{00}(\sqrt{t}u',tx',y+|x'|^{-2} x')\,
  d_{\sigma(u',x')}y \Big) dx' du'\,,
    \end{split} 
  \end{equation*}
  where we notice that $\sigma(u',x') =\sigma(u,x)$. We differentiate with
  respect to $t$, observing that, by induction in $k$, for any $l$ and    $k$   we have
\begin{equation*}
  \begin{split}
    \frac{d^k}{dt^k} t^lg(\sqrt t u',tx') =& \sum_{l_1+l_2+l_3=k}
    c_{l_1,l_2,l_3} t^{l-l_1-l_2/2}\left({u'}^{l_2}\cdot
    \nabla_u\right)^{l_2}\\
    &\left({x'}^{l_3}\cdot \nabla_x\right)^{l_3}g(\sqrt t u',tx')   \,,
  \end{split}
\end{equation*}
for any sufficiently regular function $g$ and suitable
  constants $c_{l_1,l_2,l_3}$.
From this we get
\begin{equation*}
 \begin{split}
    \left|\p^k \II_{00<}(t)\right| &\lappr_{k,M} 
    \max_{l_1+l_2+l_3=k}t^{{d}/2-1-l_1-l_2/2}  \|f\|_{k,M}  \int_{\R^n}{|u'|}^{l_2} \lan
    u'\sqrt t\ran^{-M} \\
    &\int_{B_{\sqrt2/t}(\R^{d_1})\cap B_{1/2}(\R^{d_1})}
    |x'|^{l_3-1}  
   \Big( \int_{\substack{y\in \sigma(u',x')\\|x'|^2t^2+|y+|x'|^{-2}x'|^2\le2}}
    \,
    d_{\sigma(u',x')}y \Big) dx' du'\,.
    \end{split} 
\end{equation*}
Denoting points of the space $x^\perp$ as $\bar y$, we see that the  integral over $d_{\sigma(u',x')}y$
is bounded by 
\be\label{gg1}
\int_{\substack{\bar y\in x^\perp\\|x'|^2t^2+|\bar y+\tau'|x'|^{-2}x'|^2\le2}}  1
\,d\bar y,\qquad \tau'=1-\frac12|u'|^2\,.
\ee
By \eqref{p11},  on
the support of the integrand $ |\tau'|   \le \sqrt 2 |x'| $. So there 
\be\label{gg2}
 1-\sqrt2|x'|\le\frac{|u'|^2}2\le 1+\sqrt2 |x'|  \,.
\ee
As the the domain of integration in $\bar y$ is bounded, then 
 integral \eqref{gg1} is bounded by a constant. So 
 putting $|x'|=r'$, $|u'|=R'$ and using \eqref{gg2}   we have
\begin{equation*}
 \begin{split}
\left|\p^k\II_{00<}(t)\right|\lappr_{k,M}& \max_{l_1+l_2+l_3= k}
\|f\|_{k,M}t^{{d}/2-l_1-l_2/2-1}\int_0^{1/2} {r'}^{\dd-2+l_3}\\
&\qquad\Big(\int_{\sqrt2\sqrt{1-\sqrt2 r'}}^{\sqrt2\sqrt{1+\sqrt2 r'}} \,{R'}^{n-1+l_2}\lan
{R'}^2t\ran^{-M/2} \, dR' \Big) dr'\,.
    \end{split} 
\end{equation*}
Since  $r'\le 1/ 2$, then on the domain of integration  $\sqrt{2-\sqrt2}\le R'\le \sqrt{2+\sqrt2}$, while
$
\sqrt2\sqrt{1+\sqrt2 r'}-\sqrt2\sqrt{1-\sqrt2 r'}\lappr r'  \,.
$
So the integral in  $dR'$ is bounded by $C \lan t\ran^{-M/2}
r'$. 
  Therefore 
\begin{equation*}
  \begin{split}
\left|\p^k\II_{00<}(t)\right| \lappr_{k,M} \max_{ l_1+l_2+l_3= k}
\|f\|_{k,M}t^{{d}/2-l_1-l_2/2-1}  
\lan t\ran^{-M/2}   \int_0^{1/2} {r'}^{\dd-1+l_3}\, dr'\,.
    \end{split}
\end{equation*}
This implies that for $0< t\le4$, for  any $k\le  k_*$ and any $\dd\ge 1$ we have
\be\label{int1}
|\p^k\II_{00<}(t)|\lappr_{k}\|f\|_{k,0}t^{{d}/2-k-1}\,.
\ee
While for any $t \ge4$ and any $k\le  k_*$, 
\be\label{int2}
\begin{split}
|\p^k\II_{00<}&(t)|  \lappr_{k,M} \max_{ l_1+l_2+l_3= k}
\|f\|_{k,M,{d}}t^{{d}/2- M/2-l_1-l_2/2-1}  \\
&\times 
\int_0^{\sqrt2/t} {r'}^{\dd-1+l_3} \, dr'
\lappr_{k,M}\|f\|_{k,M}  
t^{-(M+2+k +2d_1-{d})/2}  \,.
\end{split}
\ee
We recall that $\cI_{00<}(t)$ vanishes for $t<0$. 
\medskip

For $\II_{00>}(t)$ we first note  that by \eqref{p10} and \eqref{supp}    function  $\II_{00>}(t)$ vanishes  if $|t|> 4$. Next, by
 induction in $k$, we observe that 
\begin{equation}\label{eq:ind}
  \begin{split}
  \frac{d^k}{dt^k} g(tx|x|^{-2})(1-\varphi(8|x|^2/t^2)) =&
  \sum_{l_1+l_2+l_3=k} c_{l_1,l_2,l_3}|x|^{2(l_2-l_1)}t^{-3l_2-l_3}\\ \times
  &\left(\left(x\cdot
  \nabla\right)^{l_1}g\right)\,\frac{d^{l_2}}{dy^{l_2}}(1-\varphi)   \,,
    \end{split}
\end{equation}
where $c_{l_1,l_2,l_3}= 0$  if $l_3>0$ and $l_2=0$. Since
$\varphi'\neq 0$ only for 
$|t|/2\sqrt2\le|x|\le |t|/2$, then 
$$
\frac{d^{l_2}}{dy^{l_2}}(1-\varphi)t^{-3l_2-l_3}\lappr_{l_2,l_3}
|x|^{-3l_2-l_3}  \,,\quad l_2>0  \,,
$$
so that
$$
\left|\frac{d^k}{dt^k} g(tx|x|^{-2})(1-\varphi(8|x|^2/t^2)) \right|
\lappr_k |x|^{-k}\|g\|_{k,0}  \,.
$$
From here, in a way analogous to \eqref{p12}, putting again $|x|=r$ and 
$|u|=R$, we get that
  \begin{equation*}
 \begin{split}
|\p^k\II_{00>}(t)| &\lappr_{k,M} \|f\|_{k,M} 
\int_{|t|/2\sqrt 2}^{\sqrt2 }r^{\dd-k-2} \Big(\int_0^\infty R^{n-1}\lan
R\ran^{-M}\chi_{|\tau|\le \sqrt 2 r} \, dR \Big)dr
    \end{split}   
  \end{equation*}
  (here and below $\int_a^b dr =0$ if $b\le a$). 
  Since on the  integration domain, due to \eqref{supp} and    the indicator function
$\chi_{|\tau|\le \sqrt 2 r}$, we have $R^2\le 6\sqrt 2r$, then
  \begin{equation}\label{int3}
    \begin{split}
|\p^k \II_{00>}(t)| &\lappr_{k,M,n}
\|f\|_{k,M}\int_{|t|/2\sqrt2}^{\sqrt2}dr\,r^{{d}/2-k-2}\\
&\lappr_{k,M} \left\{\begin{array}{cc}
  \|f\|_{k,M}  \,,& k<{d}/2-1,\\
  \|f\|_{k,M}(1+|\ln|t||)  \,,& k\le  {d}/2-1.
\end{array}
\right.\ 
    \end{split}
  \end{equation}

If  $k< {d}/2-1$, then by the above 
 $\p^k \II_{00}(t)$ is bounded for all $t$. In this case, modifying the integrand in \eqref{p9} 
by the factor $\chi_{|x|\ge\eps}$, we see that thus obtained
functions  $\II_{00>}^\eps, \II_{00<}^\eps $ satisfy the same
estimates as the functions $\II_{00>}, \II_{00<}$ above, so the function $\II_{00}^\eps$ also does. The functions $\p^k
\II_{00}^\eps(t)$ with $\eps>0$ obviously are continuous in $t$ and
converge to  $\p^k \II_{00}(t)$ uniformly on bounded intervals. So the
latter function also is continuous. Similar  $\p^k \II_{00}(t)$ with
$k=d/2-1$ is continuous on any set $|t| \ge\eps>0$, so is continuous for
$t\ne0$.

\subsubsection{Integral $\II_{11}(t)$.}
Due to \eqref{p8} and similar to \eqref{p9}, \eqref{p99},  for any $k\le k_*$ we have 
$$
\p^k\II_{11}(t)=\int_{\R^n}\int_{|x|\geq 1/2}|x|^{-1}\,  \Big(\int_{x^\perp} d_y^k\, 
f_{11}(u,x,\bar y+\tau x |x|^{-2})[x|x|^{-2}]\,d\bar y \Big)dx du  \, .
$$
We easily see that $\II_{11}(t)$ is a $C^k$-smooth function and, since $M>{d}$ and 
 $\big|\bar y+\tau x|x|^{-2}\big|\geq |\bar y|,$ then
\be\lbl{q0}
|\p^{k}\II_{11} (t)\big|\lappr_{k,M}\|f\|_{k,M} \quad \forall t.
\ee
Now let $|t|\geq 1$. Let us write  $\p^k\II_{11}$  as 
\be\lbl{q1}
\p^k\II_{11}(t)=\int_{\R^n} \int_{|x|\geq 1/2}|x|^{-k-1} \int_{x^\perp} \Phi_k(\bar z)\,d\bar y dx du\,, 
\ee
where $\bar z=(u,x,\bar y)$, $\bar y\in x^\perp$, and 
\be\lbl{q2}
|\Phi_k(\bar z)|{ \lappr_{k} \|f\|_{k,M}}\lan \hat z\ran^{-M}, \quad\hat z=(u,x,\bar
y+\tau x|x|^{-2}). 
\ee
Obviously,
\be\lbl{q3}
|\hat z|\ge |\bar z|\,,\quad |\hat z|\ge 2^{-1/2} \big( |\bar z| + |\tau| |x|^{-1}\big).
\ee
Below we separate the cases $n\ge 1$ and $n=0$.

1) Let $n\ge 1$.

a) We first integrate in \eqref{q1} over $u$ in the spherical layer
$$
O:=\{u: |\tau| = 
\left|t-\tfrac12|u|^2\right|\le\tfrac12   t\}\,.
$$
It is empty if $t<0$, while  for $t\ge 0$,
$
O=\{u: t \le  |u|^2\le 3t\}\,.
$
By \eqref{q2} and the first relation in \eqref{q3}, for $t\ge0$  the part of the  integral in \eqref{q1} 
with $u\in O$ is bounded by
$$
K := C_k\|f\|_{k,M}   \int_O \int_{|x|\geq 1/2}|x|^{-k-1} \int_{x^\perp}
\left( |t|+|x|^2+|\bar y|^2\right)^{-M/2}\,d\bar y dx du\,. 
$$
 Since $\int_O 1\,du\le
Ct^{n/2}$, then by putting $r=|x|$, $|t|+r^2=T^2$ and $R=|\bar y|/T$ we find that 
$$
K{ \lappr_{k}}\|f\|_{k,M}t^{n/2}  \int_{1/2}^{\infty}r^{\dd-2-k}\,
T^{\dd-1-M}\,\int_0^\infty R^{\dd-2}
\left( 1+R^2\right)^{-M/2}\,dR  dr\,.
$$
The integral in $dR$ is bounded since $M>\dd$, so that
$$
K{ \lappr_{k,M}}\|f\|_{k,M}t^{n/2}\int_{1/2}^{\infty}r^{\dd-2-k}\,
\left( |t|+r^2\right)^{(\dd-1-M)/2}\,dr\,.
$$
Recalling that we are considering  the case $t\ge 1$, we put $r=
\sqrt {t}\,l$. Then
$$
K{ \lappr_{kM}}\|f\|_{k,M}t^{\frac{n+1+\dd-2-k+\dd-1-M}2}\int_{t^{-1/2}/2}^\infty
  l^{\dd-2-k} (1+l^2)^{\frac{\dd-1-M}2}\,dl  \,.
  $$
  Since $M>2\dd$, the integral over $l$ converges and we get
  $$
K { \lappr_{k,M}}\|f\|_{k,M}|t|^{-(M+2-{d}+k)/2}
 |t|^{\max(0,  k+1-\dd)/2} 
  Y(t)  \,,
$$
with $ Y=\ln t$ if $k=\dd-1$ and $ Y=1$ otherwise. 
 Then,  in the case $ Y=1$ the component 
of \eqref{q1}, corresponding to $u\in O$, is bounded by
\begin{equation}\lbl{q5}
C{ (k, M, {d})}
\|f\|_{k,M}|t|^{- \kappa}\,, \qquad \kappa=\frac{M+2-{d}}2, 
\end{equation}
 for all $ |t|\ge 1$, since $\max(0,
k+1-\dd)\leq k$. 
If $ Y=\ln t$ the same estimate holds in the case $\dd\geq 2$ since $\max(0,
k+1-\dd)< k$. In the case $\dd=1$ and $ Y=\ln t$ (i.e. $k=0$) we get \eqref{q5} with 
$\kappa$ replaced by any $\kappa'<\kappa$ (and  constant $C$ depending on $\kappa'$).

b) Now consider the integral for $u\in O^c= \R^n\backslash O$. There
$|\tau| =|t-\tfrac12|u|^2|\ge \frac12 |t|$. So, by inequalities 
 \eqref{q2} and \eqref{q3},
$|\Phi_k(\bar z) | { \lappr_{k}} \lan(u,\bar y)\ran^{-M}$ and $|\Phi_k(\bar z) |{ \lappr_{k}}(|t||x|^{-1} +|x|)^{-M}$. Let $M=M_1+M_2$, $M_j\ge 0$. Then
the part of the integral \eqref{q1} for $u\in O^c$ is bounded by
$$
C\|f\|_{k,M}   \int_{|x|\ge 1/2}
|x|^{-1-k}\left(t|x|^{-1}+|x|\right)^{-M_1}
\Big(\int_{\R^n}\int_{x^\perp} \lan(u,\bar y)\ran^{-M_2}\ d\bar y du \Big) dx\, 
$$
Choosing $M_2=n+{d_1}-1+\gamma$  with $0<\gamma<1$ (then $M_1, M_2>0$ since $M>d$) 
we achieve that the integral over
$du\,d\bar y$ is bounded by $C(\gamma)$, for any $\gamma$. Since by
Young's inequality
\footnote{Indeed, by Young's inequality with $p=1/a,\, q=1/(1-a)$ we have
	that $A^a B^{(1-a)} \le aA + (1-a)B \le C_a(A+B)$. This proves the
	assertion.}
$$
(A+B)^{-1}\le C_aA^{-a}B^{a-1}\,,\quad 0<a<1\,,
$$
for any $A,B>0$, then
$
\left(t|x|^{-1}+|x|\right)^{-M_1}\le C_a|x|^{(2a-1)M_1}|t|^{-aM_1}
$
 $ (0<a<1)$. 
So the integral above is bounded by
$$
C(\gamma)\|f\|_{k,M}|t|^{-aM_1}\int_{|x|\ge 1/2}
|x|^{-1-k+bM_1}\,dx\,,\quad b=2a-1 \in (-1, 1) \,.
$$
Denote $b_*=\tfrac{1+k-\dd}{M_1}$. Then for $b=b_*$ the exponent for
$|x|$ in the formula above equals $-{d_1}$, and $b_*>-1$  if
$\gamma$ is sufficiently small,  since $M>d$. Noting that
$$
a(b_*)M_1 = \frac{b_*+1}2M_1= \frac{M+2+k-{d} -\gamma}2 = 
 \kappa +\frac{k}2 - \frac\gamma2
$$  
($\kappa$ was defined in \eqref{q5}), 
we see that the part of 
integral \eqref{q1}, corresponding to $u\in O^c$, 
\be\label{int4}
\begin{split}
\text{ is bounded by
\eqref{q5} if $k\ge1$, while for $k=0$ it is bounded by }\\
\text{
\eqref{q5} with $\kappa$ replaced by any $\kappa'<\kappa$. }
\end{split}
\ee

2) Now let $n=0$. Then
\be\lbl{q6}
\left|\p^k\II_{11}(t)\right|\le \int_{|x|\ge 1/2}
|x|^{-1-k} \int_{x^\perp} \Phi_k(\bar z) \, d\bar y  dx \,, \quad \bar
z=(x,\bar y)\,,
\ee
where
$
|\Phi_k(\bar z) |{  \lappr_k}\lan \hat z\ran^{-M}
$
with  $ \hat z = (x,\bar y+ tx|x|^{-2})$.
Repeating literally the step 1b) above with $n=0$ we get that for
$|t|\ge 1$ the integral in \eqref{q6} may also be bounded by
\eqref{q5}. We recall that for $|t|\le 1$ the derivative
$\partial^k\II_{11}(t)$ was estimated in \eqref{q0}.

\subsubsection{Integral $\II_{10}(t)$.}
Now we use the second disintegration in \eqref{p05} instead of the first. Since
by \eqref{p8} on the support of the integrand $|y|\ge 1/\sqrt 2$, then
repeating the argument above with $x$ and $y$ swapped we get that
$\II_{10}(t)$ meets the same estimates as $\II_{11}(t)$.

\subsubsection{
End of the proof of Theorem \ref{t_ap1} 
}

Finally, 

\noindent --  combining together  relations \eqref{p13}, \eqref{int1},
\eqref{int3} and  \eqref{q0} we estimate $ \p^k\II(t)$ for $0<|t|\le4$,

while

\noindent --  combining together    \eqref{int2}, \eqref{q5}, \eqref{int4} and using  the fact that  $ \p^k\II_{00>}(t)$ and  $ \p^k\II_{00}(t)$
vanish for $|t|\ge4$ when $n=0$, 
 we estimate $ \p^k\II(t)$ for  $t\ge4$.

For the reason, explained at the end of Section \ref{s_7.3.1}, the involved derivatives are continuous functions. This proves the theorem.

\subsection{Linear transformations of quadrics}\label{s_7.4}

In this subsection we denote by $C_0$ spaces of continuous functions with compact support. 

In $\R^{d}=\{z\}$ let us consider a quadratic form with real coefficients\,\footnote{ Section \ref{s_7.4}-\ref{s_7.5}
 is the only part of our work, where 
quadratic forms  are allowed to have non-rational coefficients.}
$
F(z) = \tfrac12 Az \cdot z
$
of signature $(n_0, n_+, n_-)$ such that 
$
n_0=0, \; n_+ \ge n_-  =: {d_1}\ge1.
$
Denote $n=n_+ - n_-$. 

Using the standard diagonal normal form for a symmetric quadratic form,  we
	construct a linear transformation 
$$
L: \R^{d} \to \R^{d}, \quad z \mapsto Z=(u,x,y), \quad u\in \R^n,\; \;\;x, y\in \R^{{d_1}}, 
$$
such that 
$
Q(L(z)) = F(z)$, where 
$Q(Z) = \tfrac 12 |u|^2 + x\cdot y. 
$
Consider the corresponding quadrics 
$
\Sigma ^Q_t = \{ Z: Q(Z) =t\} $, $ \Sigma^F_t = \{ z: F(z) =t\} , 
$
and the $\delta$-measures $\mu^Q_t, \mu^F_t$  on them (e.g. see   \cite[Section~II.7]{Khin}): 
\be\label{2}
\lan \mu^Q_t,  f^Q \ran = \lim\limits_{\eps\to 0} \frac1{2\eps} \int_{ t-\eps \le Q(Z) \le t+\eps}  f^Q(Z)\, dZ,
\ee
$$
\lan \mu^F_t,  f^F \ran = \lim\limits_{\eps\to 0} \frac1{2\eps} \int_{ t-\eps \le F(z) \le t+\eps}  f^F(z)\, dz,
$$
where $f^Q, f^F \in C_0(\R^d)$ and 
 $\lan \mu, f\ran$ signifies the integral of a function $f$ against a measure $\mu$.  Then $\mu^Q_t$ and $\mu^F_t$ are Borel
 measures in $\R^d$, supported by $\Sigma_t^Q$ and $\Sigma_t^F$ respectively, and 
  for  $f^Q \in C_0\big(\Sigma_t^Q\setminus \{0\}\big)$
and $f^F  \in C_0\big(\Sigma_t^F\setminus \{0\}\big)$  we have 
$$
\lan \mu^Q_t,  f^Q \ran =  \int_{ \Sigma ^Q_t}  \frac{f^Q(Z)}{ |\nabla  Q(Z)|}\, dZ\!\mid\!_{\Sigma ^Q_t},  \quad
\lan \mu^F_t,  f^F \ran =  \int_{ \Sigma^F_t}  \frac{f^F(z)}{ |\nabla  F(z)|}\, dz\!\mid\!_{\Sigma^F_t},
$$
where $dZ\!\mid\!_{\Sigma^{Q (\text{or}\, F)}_t}$ is the volume element on $\Sigma^{Q (\text{or}\, F)}_t\setminus\{0\}$, induced from $\R^{d}$, see \cite{Khin}. Now let $f^F = f^Q \circ L$. Then the integral in \eqref{2} equals
$$
 \int_{ t-\eps \le Q(Z) \le t+\eps}  f^Q(Z)\, dZ = | \det(L)|   \int_{ t-\eps \le F(z) \le t+\eps}  f^F(z)\, dz,
$$
so passing to the limit we get that
\be\label{transform}
L \circ \big(  | \det(L)|  \mu_t^F \big) = \mu_t^Q. 
\ee
  Thus,

\noindent { to examine the function 
\be\label{I1} 
t\mapsto \cI^F(t;f)= \lan \mu^F_t, f\ran, \qquad \mu^F_t = | \nabla F(z)|^{-1}dz\!\mid_{\Sigma^F_t},
\ee
 we are free to use any linear  coordinate system in $\R^{d}$ since changing the coordinates we only modify 
 function $\cI^F$ 
  by a constant factor. }

\subsection{Sign definite forms}\label{s_7.5}
Finally let us  consider the case when $n_0=0$ and $\min(n_+, n_-)=0$, i.e. when 
the form $
F(z) = \tfrac12 Az \cdot z
$
is sign--definite and non degenerate. Suppose for definitenes 
 that $n_-=0$. Then there exists a 
linear transformation $L$ such that $F(z)=Q(L(z))$, where 
$Q(Z)=\tfrac12|Z|^2$, $Z\in \R^d$. The quadric  $\Sigma_t$ reduces to
the empty set for $t<0$, so  function $\II^F(t)$ (see \eqref{I1}) vanishes for $t<0$.
The calculation of previous subsection remains true in this case, so \eqref{transform} and 
 the change of coordinates $Z=\sqrt{2t}\,Z'$ show that
\[ \begin{split}
\II^F &(t;f)=  C(d,L) t^{-1}\int_{|Z| =\sqrt{2t} } f^Q( Z)\,\mu_{S^{d-1}_{\sqrt{2t}} }(dZ) \\
&= C(d,L) t^{d/2-1}\int_{|Z'|=1} f^Q(\sqrt{2 t} Z')\,\mu_{S_1^{d-1}}(dZ'),  \qquad t > 0,\;\  f^Q = f\circ L^{-1}
  \,,
\end{split}
\]
where $\mu_{S_r^{d-1}}$ is the volume element  on the $d-1$ sphere of radius $r$. From this relation we immediately 
get that for any $k\le \min(d/2-1,k_*)$, 
$$
\left|\p^k\II^F(t)\right| \lappr_{k}  \|f\|_{k,0}  \quad \text{if}\quad 0\le t\le1, 
$$
and 
$$
\left|\p^k\II^F(t)\right| \lappr_{k,M}  \|f\|_{k,M}
t^{-(M+2+k-d)/2}    \quad \text{if}\quad  t\ge1.  
$$

\subsection{General result}
We  sum up the obtained results in the following

\begin{theorem}\label{c_integrals}
Consider any nondegenerate quadratic form 
$
F(z) = \tfrac12 Az \cdot z
$
 on $\R^{d}$, $d\ge
3$, and   a function $f\in\cC^{k_*,M}(\R^d)$, $M>{d}$. 
Then the corresponding integral $ \cI^F(t;f)= \lan \mu^F_t, f\ran$ (see \eqref{I1}) meets the assertions of Theorem~\ref{t_ap1}. 
\end{theorem} 
\begin{proof}
i) If $n_+\ge n_-$, then by means of a linear change of variable  $F$ may be put to the normal form \eqref{p-1}, where 
$d_1\ge0$. Now  the assertion follows from the argument in Subsections \ref{s_7.4}, \ref{s_7.5}
 and Theorem~\ref{t_ap1}.

ii) If $n_- >n_+$, then the quadratic form $-F$ is as in i), and the assertion follows again since obviously
 $ \ 
\cI^{-F}(t;f) = \cI^F (-t; f) .
$
\end{proof}

\appendix
\section{The $J_0$ term: case $d=4$}
\lbl{s:appendix}
 In this section we find asymptotic for the term $J_0$ from \eqref{eq:J><} in the case 
 \be\label{dm}
 d=4\quad\text{ and } \quad m=0.
 \ee
 Below in this section we always assume \eqref{dm}.

\subsection{Preliminary results and definitions}\lbl{sec:app1}

We will need Lemmas~30 and 31 of \cite{HB}, restricted for the case $m=0$ and $d=4$, which we state below without  a proof.
 Recall that  constants $\sigma^*_\vc(A)$ are defined in \eqref{eq:sigma_p} and  $\sigma^*(A)=\sigma^*_{\mathbf 0}(A)$.
 Set $\alpha:=7/2$ and
recall~\eqref{dm}.
\begin{lemma}[Lemma~30 of \cite{HB}]\lbl{l:30}
   For any $\eps>0$ and $X\in\N$, 
\begin{equation}\label{eq:30}
\sum_{q\le X} S_q(\vc; A,0) = \eta(\vc)\sigma_\vc^*(A)\sum_{q\le X}q^{d-1} +
O_{\eps}(X^{\alpha+\eps} (1+|\vc|))  \,,
\end{equation}
 where $\eta(\vc)=1$ if $\vc\cdot A^{-1}\vc = 0$ and at the same time $\det A$ is a square of an integer, and  $\eta(\vc) = 0$ otherwise.
Moreover, $|\sigma_\vc^*(A)|\lappr_{\eps} 1+|\vc|^\eps$ when
$\eta(\vc)\neq 0$.
\end{lemma}
\begin{lemma}[Lemma~31 of \cite{HB}]\lbl{l:31}
	Assume that the determinant $\det A$ is a square of an integer. Then
	for any $\eps>0$ and $X\in\N$,
	\begin{equation}\non
	\sum_{q\le X} q^{-d}S_q(0; A,0) = \sigma^*(A)\log X + \hat C_A +
	O_{\eps}(X^{\alpha+\eps -d}) \,,
	\end{equation}
	where $\hat C_A$ is a constant depending only on $A$.
	Otherwise, if $\det A$ is not a square of an integer, then for any $\eps>0$ and $X\in\N$
	\begin{equation}\non
	\sum_{q\le X} q^{-d}S_q(0; A,0) = L(1,\chi)\prod_p (1-\chi(p)p^{-1})\sigma_p(A,0) + O_{\eps}(X^{-1/2+\eps}) \,,
	\end{equation}
	  where  $\chi$ is the Jacobi symbol $(\tfrac{\det(A)}{*})$
	and $L(1,\chi)$ is the Dirichlet $L$--function.
\end{lemma}
  We will also need the following construction.
  Let us define for $r\in\R_{>0}$
\begin{equation}\lbl{eq:I*}
I^*(r) := \tilde I_{rL}(0) =\int_{\R^{d}} w(\z)\,
h\left(r, { 
F^0}(\z)\right)
\,d\z\, .
\end{equation}
Consider a function $K(\rho;w,A)$, $\rho\in\R_{>0}$, given by
\begin{equation}\lbl{eq:def_K}
K(\rho): =\eta(0)\sigma^*(A)\left( \sigma_\infty(w;A,0)\log \rho +
\int_{\rho}^\infty r^{-1}I^*(r)\,dr\right) +\sigma_\infty(w;A,0){ \hat C_A}  \,,
\end{equation}
where  constant $\eta(0)$ is
defined according to
Lemma~\ref{l:30} and $\hat C_A$~--- according to Lemma~\ref{l:31}. Note that  functions $I^*(r)$ and $K(\rho)$ do not depend on $L$. 

We claim that  function $K(\rho)$, $\rho>0$, can be extended at $\rho=0$ by continuity. Indeed, for $0<\rho_1<\rho_2\le 1$
\be\lbl{K1-K2}
K(\rho_2)-K(\rho_1)= \eta(0)\sigma^*(A)\left(
\sigma_\infty(w;A,0)\log(\rho_2/\rho_1) - \int_{\rho_1}^{\rho_2}
r^{-1}I^*(r)\,dr\right).
\ee
Using that $I^*(r)=L^{-d} I_{rL}(0)$ (see \eqref{Iq-t_Iq}), we write the term $I^*(r)$ from \eqref{K1-K2}  in the
 form, given by  Proposition~\ref{l:I_q(0)=}\,b).  Then $I^*(r)$ takes the form of the r.h.s. of \eqref{I_q(0)=-4}, divided by $L^d$, with $q=rL$. The leading term in the obtained formula for $I^*(r)$ is $\sigma_\infty(w;A,0)$ and the corresponding integral $\int_{\rho_1}^{\rho_2} r^{-1}\sigma_\infty\,dr$ in \eqref{K1-K2}
   cancels the first term in the brackets of \eqref{K1-K2}. Then, setting $M=d/2-1,$ $\beta=r^{\bar \gamma}$,
   $\bar\gamma =\gamma/d$ and
   $0< \gamma<1$ in the just discussed formula for $I^*(r)$, obtained from \eqref{I_q(0)=-4},  we get the estimate 
\begin{equation*}
  \begin{split}
{   |}K(\rho_2)-K(\rho_1){   |}
\lappr&_{N} { \|w\|_{d/2-1,d+1}}\int_{\rho_1}^{\rho_2}
\left(r^{d/2(1-\bar \gamma)-2}\lan \log r\ran+ r^{N-2}+r^{\bar \gamma N-2}\right)dr  \\
\lappr&_\gamma \, \rho_2^{d/2-1-\gamma}{ \|w\|_{d/2-1,d+1}}  \,.
\end{split}\end{equation*}
The last inequality here is obtained by  choosing $N=N(\gamma)$ to be sufficiently large and writing 
$r^{d/2(1-\bar \gamma)-2}\lan \log r\ran \lappr_\gamma r^{d/2(1-\bar \gamma)-2 - \gamma/2} = r^{d/2 -2 - \gamma}.$
Therefore $K(\rho)$ extends at $\rho=0$ by continuity and
\begin{equation}\lbl{eq:K(rho)}
 \left| K(\rho)- K(0)\right|\lappr_{\gamma}\rho^{d/2-1-\gamma}
 { \|w\|_{d/2-1,d+1}}   \,,
\end{equation}
so the  function $K$ is $(d/2-1-\gamma)$-H\"older continuous at zero, for any $\gamma>0$. 
\subsection{Estimate for $J_0$}

Argument in this section is related to Section~13 of \cite{HB}. Here we restrict ourselves for the case when the determinant $\det A$ is a square of an integer, so in particular $\eta(0)=1$. 
We use this specification only in the proof of Lemma~\ref{l:I_B-4}, when applying Lemma~\ref{l:31}. The case of non-square determinant is easier and can be obtained similarly,  using the second assertion of Lemma~\ref{l:31}.

\begin{proposition}\lbl{l:n(0)-4} Assume that the determinant $\det A$ is a square of an integer. Then for any $0<\eps<1/5$,
  \begin{equation*}
    \begin{split}
	J_0= &\sigma_\infty(w;A,0)\sigma^*(A)L^d\log L + K(0;w,A)
        L^d \\
        &+ O_{\eps}(L^{d-\eps}
        \left(\|w\|_{d/2-1,d-1}+\|w\|_{0,d+1}\right)).
    \end{split}
  \end{equation*}
\end{proposition}
{\it Proof.}
 To establish Proposition~\ref{l:n(0)-4} we write
  $J_0$ in the form 
\eqref{eq:n(0)=}, $J_0=J^+_0+J^-_0$, where
$$
J^+_0:=\sum_{q>\rho L}  q^{-d}
S_q(0) I_q(0)\,\qmb{and}\qu J_0^-:=\sum_{q\leq \rho L} q^{-d}
S_q(0) I_q(0)\,,
$$
with $\rho\leq 1$.
Then the assertion   follows from Lemmas~\ref{l:I_A-4} and
\ref{l:I_B-4} below.
Recall that $\alpha=7/2$.

\begin{lemma}\lbl{l:I_A-4}
  Let $w\in L_1(\R^{d})$.  Then for any $\gamma>0$, any $\rho\le 1$ and $L$ satisfying $\rho L>1$,
  $$
  \left|J_0^+-L^d\eta(0)\sigma^*(A)\int_{\rho}^\infty r^{-1}I^*(r)\,dr\right|\lappr_{   \gamma}
 (\rho^{\alpha+\gamma-d-1} L^{\alpha+\gamma} + \rho^{-2}L^{d-1})|w|_{L_1}  \,. 
  $$
\end{lemma}
{\it Proof.} To simplify the notation, in this proof we denote $I_q:=I_q(0)$ and $S_q:=S_q(0)$.
Let us recall the summation by parts formula for sequences $(f_q)$ and $(g_q)$:
$$
\sum_{m<q\leq n} f_q (g_q-g_{q-1}) = f_ng_n-f_{m+1}g_{m} - \sum_{m<q<n} (f_{q+1}-f_q)g_q.
$$
We take arbitrary $R\in\N$ and apply the latter with $m=R$, $n=2R$, $f_q=q^{-d}I_q$ and $g_q=\sum_{R<
	q'\leq q} S_{q'}$, so that $g_R=0$ and $S_q=g_q-g_{q-1}$ for $q>R$.
We find
\begin{equation}\label{eq:sum_parts}
  \begin{split}
\sum_{R< q\le 2R} q^{-d}S_qI_q=& (2R)^{-d}I_{2R} \sum_{R<
  q\leq 2R} S_q\\
&- \sum_{R< q< 2R}\tilde\partial_q(q^{-d}I_q)
\sum_{R< q'\le q} S_{q'}  \,,
  \end{split}
\end{equation}
where for a sequence $(a_q)$ we denote $\tilde\partial_q a_q:=a_{q+1}-a_q$.
By \eqref{Iq-t_Iq}--\eqref{tilde_I_q},
$$
I_q=L^{d}\int_{\R^{d}} w(\z) h(q/L, { F^0}(\z))\,d\z\,.
$$
So,
\be\lbl{I_q-ests}
|I_q|\lappr \fr{L^{d+1}}{q}  |w|_{L_1} \qquad\mbox{and}\qquad
|\partial_q I_q|\lappr \fr{L^{d+1}}{q^2}  |w|_{L_1}  \,,
\ee
where the first estimate above follows from Corollary~\ref{c:h} while the second one~--- from  Lemma~\ref{l:5} with $m=1, n=N=0$.
  Then, $|\tilde\partial_q(q^{-d}I_q)|\lappr L^{d+1}q^{-d-2}  |w|_{L_1}$.
According to \eqref{eq:30} with $\eps$ replaced by $\gamma$, for $R'\leq 2R$
\be\lbl{SRR'}
\sum_{R< q\le R'} S_{q} =  \eta(0)\sigma^*(A)\sum_{R< q\le R'}q^{d-1} +
O_{\gamma}(R^{\alpha+\gamma})  \,,
\ee
where we recall that $\sigma_{\mathbf 0}^*(A)=\sigma^*(A)$. 
Let us view the r.h.s. of \eqref{eq:sum_parts} as a linear functional $G\big((S_q)\big)$ on the space of 
 sequences $(S_q)$. Then, inserting formula \eqref{SRR'} in the r.h.s. of \eqref{eq:sum_parts}, we get
\begin{equation}\lbl{SqIq}
\begin{split}
\sum_{R< q\le 2R}& q^{-d}S_qI_q  = \eta(0)\sigma^*(A) G\big((q^{d-1})\big)
\\&+ O_\gamma\Big( 
L^{d+1}|w|_{L_1}\big(R^{-d-1+\alpha +\gamma} + \sum_{R< q\le 2R}
q^{-d-2+\alpha+\gamma}\big)\Big),
\end{split}
\end{equation}
where the $O_\gamma$ term is obtained by applying \eqref{I_q-ests} together with the estimate for $\tilde\partial_q(q^{-d}I_q)$ above and replacing the sums $\sum S_q, \, \sum S_{q'}$ in the r.h.s. of \eqref{eq:sum_parts} by $O_{\gamma}(R^{\alpha+\gamma})$.
According to the summation by parts formula \eqref{eq:sum_parts} with $S_q$ replaced by $q^{d-1}$, we have 
$\sum_{R< q\le 2R}
q^{-d}q^{d-1}I_q=G\big((q^{d-1})\big).$
Thus, by \eqref{SqIq},
\be\non
\sum_{R< q\le 2R} q^{-d}S_qI_q= \eta(0)\sigma^*(A)\sum_{R< q\le 2R}
q^{-1}I_q +O_\gamma\left(L^{d+1}R^{-d-1+\alpha +\gamma}
|w|_{L_1}\right)  .
\ee
Then,
setting $R_l=\lfloor 2^l\rho L \rfloor$ we get
\begin{equation*}
  \begin{split}
J_0^+    =& \sum_{l=0}^\infty\sum_{R_l<
  q \le R_{l+1}}q^{-d}I_qS_q \\
 = & \eta(0)\sigma^*(A)\sum_{q> \rho L} q^{-1}I_q + O_\gamma\left(
 \rho^{\alpha+\gamma-d-1}L^{\alpha+\gamma}|w|_{L_1}\sum_{l=0}^\infty
    2^{-l(d+1-\alpha-\gamma)}\right)\\
    = & \eta(0)\sigma^*(A)\sum_{q> \rho L} q^{-1}I_q
    +O_\gamma\left(\rho^{\alpha+\gamma-d-1}L^{\alpha+\gamma}|w|_{L_1}\right)  \,.
  \end{split}
\end{equation*}
It remains to compare the sum $A:=\sum_{q> \rho L} q^{-1}I_q$ with the integral
$B:=L^d\int_{\rho}^\infty r^{-1}I^*(r)\,dr$. 
Since $L^dI^*(r)=I_{rL}$, then changing the variable of integration $r$ to $q=rL$, $B$  takes the form
$\int_{\rho L}^\infty q^{-1}I_{q}\,dq.
$
Then, 
\begin{equation}\lbl{int-sum}
|A-B|\leq \Big|\sum_{q> \rho L} q^{-1}I_q - \int_{\lfloor \rho L\rfloor +1 }^\infty q^{-1}I_{q}\,dq\Big| + \Big|\int_{\rho L}^{\lfloor \rho L\rfloor +1 } q^{-1}I_{q}\,dq\Big|.
\end{equation}
Due to \eqref{I_q-ests}, $|q^{-1}I_{q}|\lesssim q^{-2}L^{d+1}|w|_{L^1}$ and
$|\p_q (q^{-1}I_{q})|\lesssim q^{-3}L^{d+1}|w|_{L^1}$. 
Thus,the both terms in the r.h.s. of \eqref{int-sum} are bounded by $(\rho L)^{-2} L^{d+1}|w|_{L^1}=\rho^{-2}L^{d-1}|w|_{L^1}$.
\qed

\smallskip

Recall that $\hat C_A$ is a constant arising in Lemma~\ref{l:31}.
\begin{lemma}\lbl{l:I_B-4} 
	Assume that the determinant $\det A$ is a square of an integer. 
	Then for any $\gamma>0$, $N>1$, any $\rho\le
        1$ and $L$ satisfying $\rho L>1$, 
	\begin{equation*}
          \begin{split}
	    J_0^-=&L^{d}\sigma_\infty(w;A,0)\left(\sigma^*(A)\log(\rho L)
+\hat C_A\right)+ O_{\gamma,N}\Bigl(
              \Bigl(\rho^{\alpha+\gamma-d}L^{\alpha+\gamma} 
            \\
            & +L^d\bigl(
              \rho\log L+\rho^{N-1} +L^{1-d}\bigr)\Bigr){ \|w\|_{d/2-1,d+1}}\Bigr).
	  \end{split}
          \end{equation*}
\end{lemma}
{\it Proof.}
Inserting Proposition~\ref{l:I_q(0)=} b) with $M=d/2-1=1$ and $\beta=1$ into the
definition of the term $J_0^-$, we get 
$
J_0^-=I_A+I_B,
$
where
$$
I_A :=L^{d}\sigma_\infty(w)\sum_{q\leq \rho L} q^{-d}S_q(0), 
\qquad
I_B:= \sum_{q\leq 
  \rho L}S_q(0) q^{-d} (f_q+g_q) \,,
$$
with
\begin{align}\non
|f_q|&\lappr q L^{d-1}\left\lan\log(\frac{q
}{L})\right\ran { \|w\|_{d/2-1,d+1}} \,,\\\non
|g_q|&\lappr_{N}\left(q^NL^{d-N}+1\right)
{ Lq^{-1}\|w\|_{0,d+1}}.
\end{align} 
By Lemma~\ref{l:31},
$$\sum_{q\leq \rho L} q^{-d}S_q(0)=\sigma^*(A)\log(\rho L)
+{ \hat C_A} + O_{\gamma}((\rho L)^{\alpha+\gamma-d}).$$ 
So,
$$
I_A=L^{d}\sigma_\infty(w)\left(\sigma^*(A)\log(\rho L)
+\hat C_A\right) + O_{\gamma}(\sigma_\infty(w)
L^{\alpha+\gamma}\rho^{\alpha+\gamma-d})\,,
$$
whereas 
\be\lbl{sig_inf}
|\sigma_\infty(w)|=|\sigma_\infty(w;A,0)|=|\II({ 0})| \le \|\II\|_{0,0}\lappr_A
{ \|w\|_{0,d+1}}
\ee
on account of \eqref{yes}.  As for the term
$I_B$, since $d=4$, Lemma~\ref{l:25HB} implies that
\begin{equation*}
  \begin{split}
|I_B|\lappr\sum_{q\le \rho L}
q^{-1}(|f_q|+|g_q|)
\lappr_{N} L^d\left(\rho\log L +
\rho^{N-1}+L^{1-d}\right){ \|w\|_{d/2-1,d+1}}  \,,
  \end{split}
\end{equation*}

for $N\geq 2$.
The obtained estimates on $I_A$ and $I_B$ imply the assertion.
\qed

\smallskip

Now we conclude the proof of Proposition~\ref{l:n(0)-4}. The leading term of $J_0$ is given by the sum of leading terms from formulas for $J_0^+$ and $J_0^-$ in Lemmas~\ref{l:I_A-4} and \ref{l:I_B-4}. Since $\eta(0)=1$, it takes the form
\be\non
\begin{split}
L^d\sigma^*(A)&\Big(\int_\rho^\infty r^{-1}I^*(r)\,dr + \sigma_\infty(w)\log(\rho L)\Big) + L^d\sigma_\infty(w) 
\hat C_A \\
&=\sigma_\infty(w)\sigma^*(A)  L^d\log L   +K(0) L^{d} 
+ O_\gamma\big(L^{d}\rho^{d/2-1-\gamma}{ \|w\|_{d/2-1,d+1}}\big),
\end{split}
\ee
where in the last equality we used  \eqref{eq:def_K} and \eqref{eq:K(rho)}.
Then we find
\begin{equation*}
  \begin{split}
J_0 =& \sigma_\infty(w)\sigma^*(A)L^d\log L + K(0)
L^d+ O_{\gamma,N}\Bigl(\bigl(\rho^{\alpha+\gamma-d-1}L^{\alpha+\gamma} + \rho^{-2}L^{d-1}\\
&
+
L^d(\rho^{d/2-1-\gamma}+\rho\log
L+\rho^{N-1}+L^{1-d})\bigr){ \|w\|_{d/2-1,d+1}}
\Bigr)  \,,
  \end{split}
\end{equation*}
since $|w|_{L_1}\lappr\|w\|_{0,d+1}$.
We now pick $\rho = L^{-1/5}$ and $N=2$, and, using that $d=4$, get the assertion of proposition.
\qed

\subsection{Estimate for $\sigma_1(w;A,L)$}
\lbl{s:sigma_1}

In this section we get an upper bound for the subleading order term $\sigma_1$ of the asymptotics from Theorem~\ref{th:3-4}.

In the case when the determinant $\det A$ is not 
a square of an integer, $\sigma_1$ is given by \eqref{non-sq} and  the task is not complicated. Indeed, according to Lemma~\ref{l:31}, the product $\prod_p(1-\chi(p)p^{-1})\sigma_p(A,0)$ is finite (and independent from $L$). On the other hand, by \eqref{sig_inf}, 
$|\sigma_\infty(w;A,0)|\lappr \|w\|_{0,d+1}.$
Thus,
\be\non
|\sigma_1(w;A,L)|\lappr \|w\|_{0,d+1}.
\ee

In the case when $\det A$ is a square, $\sigma_1$ is given by  \eqref{def:sigma_1} and the required estimate is less trivial.
\begin{proposition}\lbl{l:sigma_1}
	Assume that $\det A$ is a square of an integer. Then 
	\be\non
	|\sigma_1(w;A,L)|\lappr \|w\|_{\tilde N,\tilde N+3d+4}, \quad\mbox{ where }\quad \tilde N:=d^2(d+3)-2d.
	\ee
\end{proposition}

{\it Proof.}
Since $\eta(\vc)$ takes values $0$ or $1$, then according to the definition \eqref{def:sigma_1} of $\sigma_1$, we have
\be\lbl{s_1-start}
|\sigma_1(w)|\leq |K(0)| + \sum_{\vc \ne 0:\,\eta(\vc)=1} |\sigma_\vc^*(A) \sigma_\infty^\vc(w)|.
\ee
Let us first estimate the term $K(0)$. According to \eqref{eq:K(rho)}, 
\be\lbl{K(1)-K(0)}
|K(1)-K(0)|\lappr {  \|w\|_{d/2-1,d+1}}.
\ee
On the other hand, $\sigma^*(A)$ is independent from $L$ and, in view of Lemma~\ref{l:31} is finite. Then,
by the definition \eqref{eq:def_K} of $K(\rho)$, 
$$
|K(1)|\lappr \int_1^\infty r^{-1} |I^*(r)|\, dr + |\sigma_\infty(w;A,0)\hat C_A|.
$$
Due to the definition \eqref{eq:I*} of the integral $I^*(r)$ and Corollary~\ref{c:h}, $|I^*(r)|\lappr r^{-1}|w|_{L_1}\lappr r^{-1}\|w\|_{0,d+1}$.
Then,  in view of \eqref{sig_inf}, $|K(1)|\lappr \|w\|_{0,d+1}$, so that, by \eqref{K(1)-K(0)}, 
\be\lbl{K(0)-est}
|K(0)|\lappr { \|w\|_{d/2-1,d+1}}.
\ee

Let us now estimate the terms $\sigma_\infty^\vc(w)$, which are given by \eqref{sigma_infty^c}:
$$
\sigma_\infty^\vc(w)=L^{-d}\sum_{q=1}^\infty q^{-1} I_q(\vc;A,0,L)=Y_1(\vc) + Y_2(\vc),
$$
where $Y_1=   L^{-d}
\sum_{q=1}^{L|\vc|^{-M}}q^{-1}I_q(\vc)$, $Y_2= L^{-d}
\sum_{q>L|\vc|^{-M}}q^{-1}I_q(\vc)$ and $M\in\N$ will be chosen later.
Using that $d=4$, according to Lemma~\ref{l:22},
\be\non
|Y_1(\vc)|\lappr_\ga L^{-1+\ga}|\vc|^{-1+\ga}C(w) \sum_{q=1}^{L|\vc|^{-M}} q^{-\ga}\lappr|\vc|^{-(1-\ga)(M+1)}C(w),
\ee
where we  denoted $C(w):=\|w\|_{\bar N,d+5} + \|w\|_{0,\bar N+3d+4}.$
On the other hand, by Proposition~\ref{l:19}, $|I_q(\vc)|\lappr_N L^{d+1}q^{-1}|\vc|^{-N}\|w\|_{N,2N+d+1}$ for every $N\in\N$. So, 
\be\non
|Y_2(\vc)|\lappr_N L|\vc|^{-N}\|w\|_{N,2N+d+1} \sum_{q>L|\vc|^{-M}}q^{-2} \lappr |\vc|^{-N+M} \|w\|_{N,2N+d+1}.
\ee
Thus,
\be\non
|\sigma_\infty^\vc(w)|\lappr_{\ga,N} \big(|\vc|^{-(1-\ga)(M+1)} + |\vc|^{-N+M}\big)
\big(\|w\|_{\bar N,\bar N+3d+4} + \|w\|_{N,2N+d+1}\big).
\ee  
By Lemma~\ref{l:30}, $|\sigma_\vc^*(A)|\lappr_\ga 1+ |\vc|^\ga$ if $\eta(\vc)=1$, so we get
\be\non
\sum_{\vc \ne 0:\,\eta(\vc)=1} |\sigma_\vc^*(A) \sigma_\infty^\vc(w)|\lappr_{\ga,N}
\|w\|_{\bar N,\bar N+3d+4} + \|w\|_{N,2N+d+1},
\ee
once $M$ and $N-M$ are sufficiently large and $\ga$ is sufficiently small. 
Choosing $M=d$, $N=2d+1$ and $\ga=1/(d+3)$, we get $\bar N=d^2(d+3)-2d$.
Together with \eqref{s_1-start} and \eqref{K(0)-est}, this implies the assertion of the proposition.

\section{Constants $\sigma(A,0)$ and $\sigma^*(A)$}
It is clear that our result provides an
		approximation to the series $N_L(w;A,m)$ through the singular
		integral $\sigma_\infty(w)$ only if the singular series $\sigma(A,m)$ or 
		$\sigma^*(A)$ are strictly positive. In fact,  the
		singular series is known to be strictly positive under a very
		general condition, namely, for non-singular forms of any
	degree that have non-singular solutions in $\R$ and in every p-adic  field (provided the singular series is absolutely convergent), see,
	e.g., Section 7 of \cite{Bi}. However, since the most interesting case in
	applications to mathematical physics is the case of the quadratic
	form $F_d(x,y)$ below, we give in this Appendix a direct elementary
	treatment of  the evaluation of the constants
		$\si(A, 0)$ for $d\geq 5$ and $\si^*(A)$ for $d=4$ in this case, independent of the general theory.

 In this section we consider the case when the quadratic form reads as
\be\lbl{F=F_d}
F(x,y)=\Sigma_{i=1}^{d/2} x_iy_i =: F_d(x,y)\qquad \mbox{where}\qquad  d=2s\ge 4
\ee
and  $x=(x_1,\dots,x_{s})$, $y=(y_1,\dots,y_{s})$.
Our goal is to evaluate the constants $\si(A, 0)$ for $d\geq 5$ and $\si^*(A)$ for $d=4$. Below we 
 use the usual notation for the relation that an integer $m$ 
divides  or non-divides an integer vector $s$ (e.g. $2 | (8,6)$ and $2\nmid (8,7)$).

In view of the definitions \eqref{eq:sigma_p}--\eqref{77}, our first aim is to compute the constants $\si_p(A,0)$. For a prime $p$ and $k\in\N$ let consider the set 
$$
S_p(k)=\{ (x,y) \!\!\! \mod p^{k}:\ F_d(x,y) = 0\!\!\! \mod p^{k}  \}
$$ 
and denote $ N_p(k):=\sharp S_p(k)$. Note that the set $S_p(k)$ and the constant $N_p(k)$ depend on $d$.
Then the constants $\si_p$ can be rewritten as
\be\label{sigma}
\si_p(d):=\si_p(A,0)=\lim_{k\to \infty}\frac{  N_p(k) }{p^{(d-1)k}}.
\ee
  This relation is mentioned in  \cite{HB}, p.\,199,
   without a proof; we  sketch  its rigorous derivation at
   the end of this appendix.
 
 Let $ {\mathcal N}_p(d):= N_p(1)$ be the number of
 $\F_p$--rational points on $\{F_d=0 \mod p\}$.

\begin{lemma}\lbl{l:s_p}
For any prime~$p,$
\begin{equation}\label{sip} \si_p(d) =\frac{ {\mathcal N}_p(d) -1 }{p^{d-1}-p^{1-d}}.
\end{equation}
\end{lemma}
{\it Proof.}  
For $j=0,1, \dots, k$ we define $S_p(k,j)$ as a set of $(x,y) \in S_p(k)$ such that 
$$
(x,y)  = p^j (x',y')\! \!\mod  p^k, \; \; \text{where} \;\;\;
p \nmid (x',y').
$$
So
$
S_p(k,0) = \{  (x,y) \in S_p(k):  p \nmid  (x,y) \}
$
and 
$
S_p(k,k) = \{(0,0)\}. 
$
Sets $S_p(k,j) $ and $S_p(k,j') $ with $j\ne j'$ do not intersect, and denoting $ {N}_p(k,j)=\sharp  {S}_p(k,j)$ we have 
$$
S_p(k)={\bigcup}^{k}_{j=0} {S}_p(k,j),\quad N_p(k)={\sum}^{k}_{j=0} {N}_p(k,j)\,.
$$  
In particular, 
$
 {N}_p(1,0)={\mathcal N}_p-1$ since $
  {N}_p(1,1)=1.
$
 We claim that
$$ 
{N}_p(k,0)= {N}_p(k-1,0) p^{(d-1)},
$$
and thus 
\be\lbl{Np(k,1)}
 {N}_p(k,0)= {N}_p(1,0) p^{(d-1)(k-1)}=\left({\mathcal N}_p-1\right) p^{(d-1)(k-1)}.
\ee
 Indeed, we argue by induction in $k$.  Let $k=2$ and $(x,y)  \in  {S}_p(2,0)$. Let us  write $(x,y)$ 
 as $(x_0+pa, y_0+pb)$  with $(x_0, y_0),\, (a,b) \in \F_p^d$. Then $p \nmid (x_0,y_0) $, so $(x_0, y_0)\in  {S}_p(1,0)$.
  Let  us now fix any $(x_0, y_0)\in  {S}_p(1,0)$ and look for $(a,b)\in  \F_p^{d}$ 
 such that $(x_0+pa, y_0+pb)\in  {S}_p(2,0)$. 
 Since    $p^2 F(a,b) =0 $ mod $p^2$ and 
  $p \nmid (x_0,y_0) $, then relation $F(x,y) =0 $ mod $ p^2$ implies 
     a    non-trivial  linear equation on $(a,b) \in \F_p^d$. 
 So each $(x_0,y_0) \in S_p(1,0)$ 
 generates exactly $p^{d-1}$ vectors $(x,y)\in  {S}_p(2,0)$, which  proves the formula for $k=2$.  This argument remains valid for any $k\ge 2$,
 by representing $(x,y) \!\!\! \mod p^{k}$ in the form $(x_0+p^{k-1}a, y_0+p^{k-1}b)$  with $(x_0, y_0)\in\F_{p^{k-1}}^d$ and $(a,b) \in \F_p^d$.

 Let now $(x,y)\in  {S}_p(k,j)$ with $j\geq 1$. Then $(x,y)=p^{j}(x',y') \! \!\mod  p^k$, where $p \nmid  (x',y')$ and $(x',y')$
  satisfies $p^{2j}F(x',y')=0 \mod p^k$. Thus  $(x',y')\in  {S}_p(k-2j,0)$, if $j\le \frac{k-1}{2}$,  i.e. 
  $
  j\le \lfloor \frac{k-1}{2}\rfloor = : j_k.
  $
  The correspondence $(x,y) \mapsto (x',y')$ is a well defined mapping from  $ {S}_p(k,j)$ to $ {S}_p(k-2j,0)$. Indeed, if $(x_1, y_1)\sim(x,y)$ in
   $ {S}_p(k,j)$, then $p^{k-j} |   \big((x'_1, y'_1) - (x',y')\big)$, so  $(x'_1, y'_1) \sim (x',y')$ in $ {S}_p(k-2j,0)$. Since this map is obviously 
   surjective, then it is   
   a bijection of $ {S}_p(k,j)$ onto $ {S}_p(k-2j,0),$ which in view of \eqref{Np(k,1)}     implies
$$
 {N}_p(k,j)= {N}_p(k-2j,0)=\left({\mathcal N}_p-1\right) p^{(d-1)(k-2j-1)}.
$$ 
By \eqref{Np(k,1)} this formula  as well holds for $j=0$.

Any $(x,y)$  such that \, $p^{j} | (x,y)$  with $j\ge j_k+1$ 
satisfies  $F(x,y) = 0 \!\!\mod p^{k}$. Thus 
$$
\sum_{j={   j_k+1}}^{k}\! {N}_p(k,j)= \sharp\{(x,y) \!\!\!\mod p^k\!: (x,y) = 0\! \!\!\!\mod p^{{j_k+1}}\}= p^{d(k-{ j_k} -1)}\le p^{dk/2}.
$$
Therefore
$$
N_p(k)=\left({\mathcal N}_p-1\right)p^{(d-1)(k-1)} {\sum_{j=0}^{{j_k}}}p^{-2 j (d-1)}+O(p^{dk/2}).
$$
So
$$ \si_p= \lim_{k\to \infty}\frac{N_p(k)}{p^{(d-1)k}}=\left({\mathcal N}_p-1\right) p^{1-d } \sum^{\infty}_{j=0} p^{-2 j(d-1)} 
 =\frac{p^{1-d } ({\mathcal N}_p-1 ) }{1-p^{2 -2d}} \,,
$$
which proves \eqref{sip}.
\qed
\smallskip

Let then deduce a  formula for ${\mathcal N}_p(d)$ using induction in $d/2=s$. 
	For $d=2$ we have
	$
	{\mathcal N}_p(2)=\sharp\{(x,y) \in \F_p^2 : xy=0 \!\!\!\ \mod p\}=2p-1$. 
Next, 
\[\begin{split}
{\mathcal N}_p(d+2)=\sharp \{ \hbox{\rm solutions with }   x_{s+1}=0\}+\sharp \{ \hbox{\rm solutions with } x_{s+1}\ne 0\}\\
 =p{\mathcal N}_p(d )+(p-1)p^{d}. \hskip 5 cm
\end{split}
\]
Therefore for any even $d=2s\ge2$, 
$$
{\mathcal N}_p(d)=p^{d-1}+p^{s} -p^{s-1},
$$ 
and thus 
$$\; \si_p(d)=\frac{1+p^{1-s}-p^{-s}-p^{1-d}}{1-p^{2-2d}}=\frac{(1+p^{1-s})(1-p^{-s})}{1-p^{2-2d}}  .$$
Since by Euler's formula
$
\prod_p (1-p^{-l}) = 1/\zeta(l)$ for any $l>1
$, 
then in the case $d=4$ we get from \eqref{77} and the obtained formula for $\si_p(d)$ that 
$$
 \si(A,0;d=4)=\prod_{  p}\si_p(4)=\frac{\zeta(6)}{\zeta(2) }\prod_{  p}\big(1+p^{-1}  \big).
$$
This 
 does not converge, but
$$\si^*{(A;d=4)}=\prod_{  p}(1-p^{-1})\si_p(4)=\frac{\zeta(6)}{\zeta(2)^2}=\frac{4\pi^2}{105}\simeq  0.376,$$ 
converges. Further, 
$$\si{(A,0;d=6)} =\frac{\zeta(2)\zeta(10) }{\zeta(3)\zeta(4)}\simeq 1.265,
\quad \si{(A,0;d=8)} =\frac{\zeta(3)\zeta(14) }{\zeta(4)\zeta(6)}\simeq 1.092,$$
whereas
$$
  1<\si{ (A,0;d)} =\frac{\zeta(s-1)\zeta(2d-2)
  }{\zeta(s)\zeta(d-2)}=\frac{(1+2^{1-s})(1+2^{2-4s})}{(1+2^{-s})(1+2^{2-2s})}
  + o(1)=1+ o(1)
$$
tends to 1 when $d=2s\ge 10$ grows.

  It remains to prove \eqref{sigma}. By definition
   \eqref{eq:sigma_p}, $\si_p=\sum_{t=0}^\infty p^{-dt}S_{p^t}(\bf{ 0})$,
where
$$S_{p^t}({\bf{0}})={\sum_{a\, \mathrm{mod}\,p^t}}^*\sum_{ {\bf{b}}\, \mathrm{mod}\,p^t} e_{p^t}(aF(\bf{b})) .$$
Note that $ p^{-dt}S_{p^t}({\bf{ 0}}) =1$  for $t=0$, while
for $t=1$: 
\begin{equation*}
  \begin{split}
    p^{-d}S_{p}({\bf{ 0}})&=p^{-d}\sum_{a=1 }^{p-1}\sum_{ {\bf{b}}\, \mathrm{mod}\,p} e_{p}(aF({\bf{b}})) \\
 & = p^{-d}\sum_{a=1 }^{p-1}\sum _{ {\bf{b}}\, \mathrm{mod}\,p,
  \,p|F({\bf{b}}) }1+p^{-d}\sum_{a=1 }^{p-1}\sum _{{\bf{b}}\,
  \mathrm{mod}\,p,\,p\nmid F({\bf{b}})  }e_{p}(aF({\bf{b}}))\\
&=p^{-d}(p-1){\mathcal N}_p(d)+p^{-d}(-1)(p^d-{\mathcal
      N}_p(d))=p^{1-d} {\mathcal N}_p(d)-1\ ,
  \end{split}
\end{equation*}
since 
\begin{equation}\label{ep}\sum_{a=1 }^{m-1}e_{m}(an)=-1\ ,\end{equation} 
for any $n,m\neq 0$ such that $(m,n)=1$. Therefore
$ \sum_{t=0}^1 p^{-dt}S_{p^t}({\bf 0}) =p^{1-d}  {  N}_p(1).$

We proceed now by induction, supposing that, for $k\ge 1$,
$$
\sum_{t=0}^k p^{-dt}S_{p^t}({\bf 0}) =p^{(1-d)k}  {  N}_{p}(k)\ .
$$
Then we write
$$  S_{p^{k+1}}({\bf {0}}) = {\sum_{a\, \mathrm{mod}\,p^{k+1}}}^*\sum_{{\bf b}\, \mathrm{mod}\,p^{k+1}} e_{p^{k+1}}(aF({\bf b})) =  \Sigma_1 + \Sigma_2+  \Sigma_3\ ,$$
where we have defined
\begin{equation*}
  \begin{split}
    \Sigma_1&:={\sum_{a\, \mathrm{mod}\, p^{k+1}}}^{\!\!\!\!\!*} \sum_{ p^{k+1}|F({\bf{b}})}1 = p^{k}(p-1)  {
  N_{p}(k+1)}\ ,\\
    \Sigma_2&:={\sum_{a\, \mathrm{mod}\,
    p^{k+1}}}^{\!\!\!\!\!*}\sum_{F({\bf{b}})=lp^k}
    e_{p }(al) = 
-p^{k}(p^d{  N}_{p}(k)-{  N}_{p}(k+1)    )\ ,\\
    \Sigma_3&:= {\sum_{a\, \mathrm{mod}\,
    p^{k+1}}}^{\!\!\!\!\!\!*} \,\sum_{s=0}^{k-1}\sum_{F({\bf{b}})=lp^s}
    e_{p^{k+1-s} }(al) = 0\ ,
  \end{split}
\end{equation*}
  with a non-zero
   $l =l(b) $ such that $p\nmid l$. The equalities above
  essentially  follow by a repeated application of \eqref{ep}.

This way we have got
$$
\frac{S_{p^{k+1}}({\bf {0}})}{ p^{d(k+1)}} =\frac{p^{k+1}
 {N}_{p}(k+1) -  p^{d+k}  {  N}_{p}(k)  }{ p^{d(k+1)}}
=\frac{  {N}_{p}(k+1)}{p^{(d-1)(k+1)}} -
\frac{  {N}_{p}(k)}{p^{(d-1)k}}\ ,
$$
which completes the induction step, thus proving \eqref{sigma}.

\end{document}